\documentclass[a4paper, abstracton]{scrarticle}

\usepackage[utf8]{inputenc}
\usepackage{lmodern}
\usepackage[T1]{fontenc}
\usepackage[english]{babel}
\usepackage{amsmath,amssymb, amsthm}
\usepackage{xcolor}
\usepackage{tikz}
\usetikzlibrary{arrows}
\usetikzlibrary{positioning}
\usepackage{graphics}
\usepackage[normalem]{ulem}
\usepackage{todonotes}
\usepackage{hyperref}
\usepackage{algorithm}
\usepackage{bm}
\usepackage[indLines=false]{algpseudocodex}
\usepackage{multirow}
\usepackage[onehalfspacing]{setspace}
\usepackage{caption}
\usepackage{subcaption}

\usepackage{natbib}

\newtheorem{theorem}{Theorem}

\usepackage{authblk}

\newcommand{\X}{{\mathcal{X}}}

\newcommand{\FIspace}{\ensuremath{\Phi_I}} 
\newcommand{\FI}{\ensuremath{\bm{\phi}_I}} 
\newcommand{\FSspace}{\ensuremath{\Phi_S}} 
\newcommand{\FS}{\ensuremath{\bm{\phi}_S}} 

\newcommand{\px}{{\bm{x}}}

\newcommand{\feat}{{f}}
\newcommand{\dash}{---}

\begin{document}

\title{Feature-Based Interpretable Surrogates for Optimization}

\author[1]{Marc Goerigk}
\author[2]{Michael Hartisch}
\author[1]{Sebastian Merten\thanks{Corresponding author. Email: sebastian.merten@uni-passau.de}}
\author[3]{Kartikey Sharma}

\affil[1]{Business Decisions and Data Science, University of Passau,\authorcr Dr.-Hans-Kapfinger-Str. 30, 94032 Passau, Germany}
\affil[2]{Analytics \& Mixed-Integer Optimization, University of Erlangen–Nuremberg,\authorcr Cauerstraße 11, 91058 Erlangen, Germany}
\affil[3]{Interactive Optimization and Learning Laboratory, Zuse Institute Berlin,\authorcr Takustr. 7, 14195 Berlin, Germany}

\date{}

\maketitle

\begin{abstract}
For optimization models to be used in practice, it is crucial that users trust the results. A key factor in this aspect is the interpretability of the solution process. A previous framework for inherently interpretable optimization models used decision trees to map instances to solutions of the underlying optimization model. Based on this work, we investigate how we can use more general optimization rules to further increase interpretability and, at the same time, give more freedom to the decision-maker. The proposed rules do not map to a concrete solution but to a set of solutions characterized by common features. To find such optimization rules, we present an exact methodology using mixed-integer programming formulations as well as heuristics. We also outline the challenges and opportunities that these methods present. 
In particular, we demonstrate the improvement in solution quality that our approach offers compared to existing interpretable surrogates for optimization, and we discuss the relationship between interpretability and performance. These findings are supported by experiments using both synthetic and real-world data.
\end{abstract}


\section{Introduction}\label{sec::intro}
\subsection{Motivation}

The widespread availability of easy-to-use, off-the-shelf machine learning tools has dramatically expanded their use. Yet users often perceive these tools as black boxes. Consequently, the importance of interpretability and explainability
has increased significantly in recent years, especially in areas where large datasets form the basis. For optimization problems, however, only recently have researchers started to aim for more interpretable or explainable approaches. One reason for this delay might be that, despite the availability of solvers, expert knowledge is still crucial for appropriately modeling optimization problems. These experts possess a deep understanding of both the modeling and solution techniques involved. Consequently, there has been little intrinsic motivation for mathematicians and optimization specialists to prioritize more interpretable or explainable solution techniques.

However, this perspective is not shared by all stakeholders especially those who must accept and implement the optimization results. For instance, workers who are directly affected by optimized decisions, such as those scheduled or guided by the outcome of an optimization process. They need to understand these decisions so as to help them prepare for future outcomes, e.g., potential overtime or subsequent planning decisions. The lack of transparent decision processes or explanations can lead to discontent and loss of trust, resulting in poor adoption of optimized decisions and ultimately rendering the optimization process ineffective.
Interpretable solution approaches also allow users to adapt to changing circumstances on the fly, leading to better outcomes.

Therefore, to encourage acceptance of actions based on optimization programs and build trust in their outcomes, clear, practical, and easily understandable explanations are essential. Justifying \textit{why} a system recommends certain actions is more important to stakeholders than explaining \textit{how} the solution is obtained~\citep{ji2000use,ye1995impact}. Hence, providing the underlying optimization model may not be as useful since ``disclosure alone does not guarantee that anyone is paying attention or is able to accurately interpret the information; more complex information is more likely to be unexamined or misunderstood''~\citep{prat2005wrong}.

Moreover, offering adequate information about the company, its decisions, and the worker's role within it can enhance their identification with the company~\citep{smidts2001impact}. To this end, effective and transparent communication is crucial in fostering trust and positive relationships between employees and the organization~\citep{rawlins2008measuring,men2014effects,yue2019bridging,doi:10.1287/mnsc.2017.2906}.
Conversely, reliance on complex and hard-to-understand decision models can foster skepticism and a reluctance to adopt decision-support systems, even if these models have been shown to improve decision-making performance~\citep{arnold2006differential,kayande2009incorporating}. 

Therefore, it is essential for managers and decision-makers to provide insights into decisions derived from optimization processes. As a first step, incorporating workers into the feedback cycle of the modeling process is crucial, as providing explanations during this process enhances the feedback loop~\citep{chakraborti2019plan}. 
Already in this phase, interpretable surrogates unlock several practical benefits beyond transparency: they make model faults (e.g., missing constraints, miscalibrated thresholds, unintended biases) immediately apparent to non-experts, simplifying debugging, and they enable straightforward counterfactual explanations by showing which minimal feature changes would alter a recommendation.
However, even after a model is agreed upon, stakeholders still need clear and easily understandable explanations of when and why different solutions are considered. This is evident in decision support systems, where decision-makers or stakeholders must explain outcomes, such as in loan approval processes~\citep{sachan2020explainable,strich2021world} or medical consultations~\citep{scott1977explanation,swartout1985explaining,lamy2019explainable}. Ultimately, the European ``right to explanation'' ensures that users can request an explanation for any algorithmic decision made about them~\citep{goodman2017european}.

The main contribution of this paper is not only to generate interpretable surrogates for optimization, but to explore how their inputs and outputs should be presented. We argue that, in practice, conveying only high-level solution features---rather than a fully specified solution---better aligns with human decision‐making. For instance, in the shortest-path setting, route-choice studies show that travelers rely on a handful of landmarks to guide their routes, rendering strict optimal paths poor predictors of actual behavior \citep{manley2015heuristic,manley2015shortest}. This insight suggests that feature-based recommendations support intuitive decision-making and that offering suboptimal yet easily comprehensible guidance aligns with real-world practice, since humans rely on heuristics and seldom pursue strictly optimal solutions. Moreover, by prescribing a set of essential features (e.g., ``pass through landmarks A, B, and C'', ``avoid city center''), we leverage users’ domain knowledge to fill in the implementation details, reducing the need for micromanagement. Such flexibility also accommodates minor environmental changes: the features can flexibly map to slightly different realizations without compromising the outlined solution scope. In contrast, prescribing a single rigid solution risks being overly constraining---potentially leading users away from near-optimal outcomes---whereas solution features strike a balance between guidance and adaptability. Similarly, representing a new instance using features is often more practical in real-world settings. Beyond this convenience, observable features---such as seasonality or weather---can significantly influence decisions, and omitting them can lead to systematically biased or inconsistent outcomes~\citep{ban2019big}.

With this mindset, our goal is to provide straightforward, easily comprehensible, feature-based rules that make it possible to deduce a result from the instance itself. In other words, we are aiming for an interpretable surrogate for the optimization process that indicates meaningful solution features while empowering users to determine the precise implementation. This differs from explaining a solution, which merely requires a post hoc justification of the result. It is important to note that any interpretable approach is inherently explainable, as one can simply reference the rule used to derive the solution. However, the reverse is rarely true: explainable approaches usually are not interpretable~\citep{rudin2019stop}.

In this paper, we rethink the framework presented in~\citet{GOERIGK20231312}. There, the authors aimed to find comprehensible optimization rules that map cost scenarios to solutions. However, this approach has the drawback of not allowing the user to deviate from the provided solution, making it difficult to adapt to unexpected situations. Furthermore, when describing an instance or a solution comprehensibly in practice, one typically focuses on the most relevant features rather than detailing every minor aspect. Hence, in this paper, we extend the interpretable optimization framework to make use of features. In particular, any given vector of instance features is mapped (in a comprehensible way) to a vector of solution features. We call the corresponding set of solutions that exhibit these features a meta-solution. Figure~\ref{fig:summary} provides a schematic comparison of our approach and the related procedures used in data-driven optimization.

\begin{figure}[htb]
    \centering
    \includegraphics[width=\linewidth]{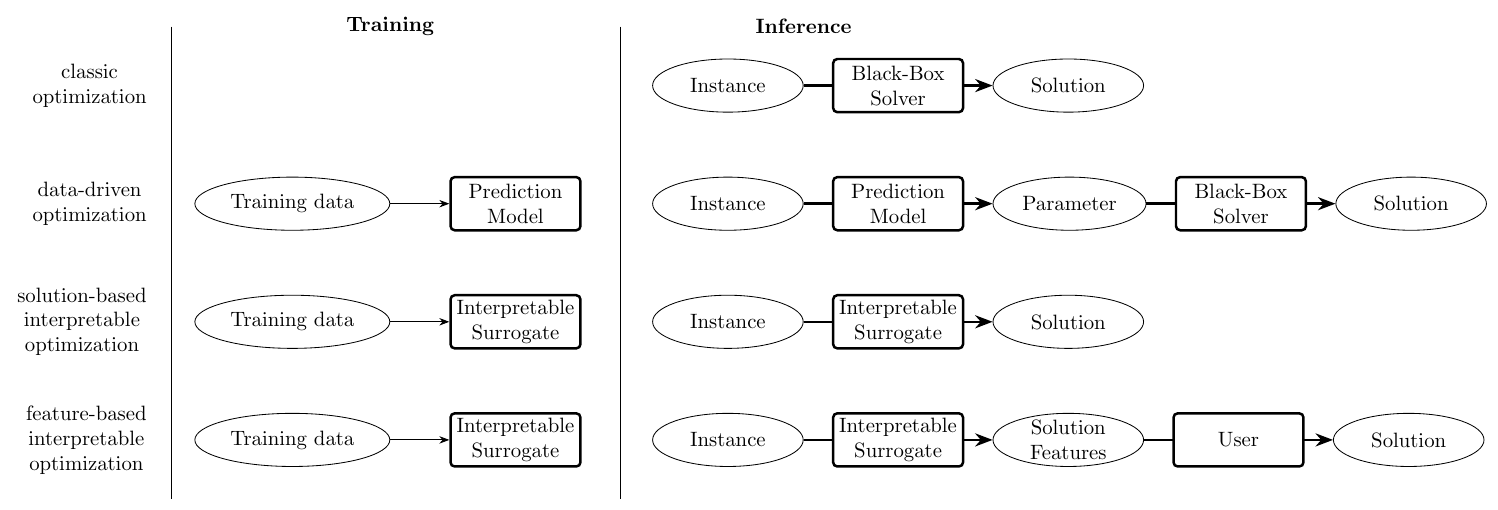}
    \caption{
Comparison of the proposed feature-based framework with related approaches for solving optimization problems. In our method, the traditional use of a black-box solver is replaced by an interpretable surrogate model, which is first trained on available data. For a new problem instance, the surrogate specifies a restricted solution space using comprehensible features, enabling the user to select a solution suitable for their situation.}
    \label{fig:summary}
\end{figure}

\subsection{Related Literature}
Interpretability and explainability have only recently gained attention in the field of optimization. In contrast, the broader field of artificial intelligence (AI) has seen a rapid increase in interest regarding explainable AI~\citep{adadi2018peeking} and more interpretable approaches~\citep{rudin2019stop} over the last decade. This shift in AI underscores the importance of transparency and trust in complex systems, a principle that is equally vital in optimization.

To provide context, we first discuss what constitutes ``easily comprehensible'' descriptions, acknowledging the subjectivity of this terminology, before we present literature aimed at enhancing comprehensibility. Several studies have investigated which representation types are most comprehensible for human users~\citep{miller2019explanation, arzamasov2021comprehensible}. These studies suggest that decision trees and classification rules~\citep{baehrens2010explain, huysmans2011empirical, freitas2014comprehensible}, as well as linear models~\citep{setiono1997neurolinear, parker2015evaluating, ustun2016supersparse}, are particularly understandable, with sparsity being a significant factor contributing to their comprehensibility~\citep{gleicher2016framework}. To achieve comprehensible mappings from instances to solutions as well as to balance quality and interpretability, several approaches have been proposed. These include post hoc simplification to improve interpretability~\citep{nguyen2012computational}, enforcing linear representations~\citep{branke2015hyper}, and semantically constraining the generated rules~\citep{hunt2016genetic}. Overall, most approaches claim interpretability based on the reduced size of the resulting rules and expressions.

\paragraph{Interpretable optimization approaches.}
Interpretability calls for an easily comprehensible mapping from instances to solutions. In this sense, linear programming can be considered an interpretable framework: given an instance, one can apply the simplex algorithm (an ``easily comprehensible'' rule) to find the solution. However, as previously discussed, the ability to comprehend such methods greatly depends on the individual's domain knowledge. Therefore, more generally comprehensible mappings are needed to achieve broader interpretability. This necessitates optimization rules that balance quality and interpretability. Consequently, the goal of obtaining interpretable optimization rules can be viewed as finding easily comprehensible heuristics. This approach is closely linked to the field of generation hyper-heuristics~\citep{drake2020recent}.
Early approaches in this field already aimed to ensure interpretability by creating representations with limited size~\citep{burke2007automatic, nguyen2017pso}.
Various strategies have been developed to harmonize quality and interpretability, such as simplifying heuristics post hoc~\citep{nguyen2012computational}, using linear representations for clarity~\citep{branke2015hyper}, and imposing semantic constraints on the rules~\citep{hunt2016genetic}.
Recent applications of hyper-heuristics that also strive for interpretability include routing policies~\citep{wang2019novel}, dispatching rules~\citep{ferreira2022effective}, and online combinatorial optimization problems~\citep{zhang2022deep}. Note that most of these approaches to generate heuristics use genetic programming~\citep{burke2019classification} or other search algorithms, which stands in contrast to the optimization approach presented in this paper.

Furthermore, approaches tailored for specific settings have been introduced, such as a method for stopping in stochastic systems~\citep{doi:10.1287/mnsc.2020.3592} and a mixed integer programming approach to obtain interpretable tree-based models for clustering~\citep{bertsimas2021interpretable}. However, these methods only assign class labels (stop/continue, cluster membership) in an interpretable manner and do not apply to general optimization settings.

In contrast, the framework proposed in this paper provides an interpretable optimization rule that serves as a surrogate for the optimization process and can be applied to any optimization problem. It allows for easily comprehensible mappings from instances to understandable solution spaces described by features. This approach generalizes the method presented in~\citet{GOERIGK20231312}, which had the limitation of only being able to map to concrete solutions, thus reducing flexibility when encountering new scenarios. In the rest of this work, we will also refer to this framework as the \textit{solution-based approach} and to optimal decision rules generated using it as \textit{solution-based interpretable surrogates}. Other works, which pursue more general concepts, deal with the generation of decision trees which map to solution strategies \citep{bertsimas2021voice} and the robustification of the solution-based approach \citep{goerigk2024towards}.

\paragraph{Explainable optimization.}

In contrast to interpretability, explainability involves post hoc justifications for why a particular solution was chosen without necessarily providing true insights into the underlying mechanism. In optimization, necessary and sufficient optimality conditions and sensitivity analysis can serve as forms of explanation. Additionally, contrastive explanations can help clarify why the found solution is optimal. These explanations often require a mathematical background and a deep understanding of the optimization model.

To address this complexity, various methods for enhancing explainability in optimization have emerged in recent years, drawing inspiration from AI techniques. In~\citet{de2023explainable}, the authors outline how AI methods solving operations research problems should be used responsibly, also considering explainability as one critical aspect. To advance explainable optimization, several domain-specific approaches evolved, such as argumentation-based explanations for planning problems~\citep{collins2019towards, oren2020argument}, scheduling problems~\citep{vcyras2019argumentation, vcyras2021schedule}, and workforce allocation~\citep{poveda2024trustworthy}. Later, more general methods were developed for dynamic programming~\citep{erwig2021explainable}, multi-stage stochastic optimization~\citep{tierney2022explaining}, and linear optimization~\citep{kurtz2024counterfactual}. Furthermore, data-driven explainability frameworks were introduced in~\citet{aigner2023framework} and~\citet{forel2023explainable}, where the importance of selecting suitable features is a crucial aspect \citep{aigner2025feature}. Additionally, for the case of multi-objective optimization problems, where interaction with a decision maker is crucial to balance different criteria and reach the best compromise on the Pareto frontier, several methods have been proposed to enhance explainability~\citep{sukkerd2018toward, misitano2022towards, corrente2024explainable, misitano2024exploring}.

\paragraph{Other related research.}

Optimization under uncertainty is highly relevant to our work, as we aim for an optimization rule that has to work well for any (unknown) data that we might observe in the future. In particular, problems with a two-stage setting are closely related. In our setting, the first stage consists of finding an appropriate optimization rule, such that for the anticipated scenarios, the realized solution in the second stage (that has to adhere to the optimization rule) minimizes the overall decision criteria, which can be for example, the worst-case or the Laplace criteria. A very close connection can be drawn to so-called decision rules~\citep{bertsimas2015design,georghiou2019decision}, where in a multistage setting, decisions that need to be taken in later stages are merged into a decision rule, which has to be determined in the first decision stage. In particular, a decision rule specifies the reaction to a scenario and can, therefore be viewed as an explanation for how a scenario affects a decision. However, the primary motivation for using decision rules is not to enhance the comprehensibility of decisions in later stages but to obtain approximate decisions in a reasonable time. The selection of function classes for decision rules is based mainly on computational performance and approximation quality. Also, in the realm of optimization under uncertainty, a connection can be drawn to $K$-adaptability~\citep{buchheim2017min, hanasusanto2015k,malaguti2022k}, where the best $K$ solutions are selected in the first decision stage, with the decision of which solution to choose in a specific scenario being deferred until the scenario is revealed. Notably, there is no rule provided for mapping scenarios to these solutions, and furthermore, these solutions are static, in contrast to our setting of meta-solutions. Only in recent studies has a connection been explored between $K$-adaptability and decision rules~\citep{subramanyam2020k, vayanos2020robust}, but these efforts have prioritized improving computational properties rather than enhancing comprehensibility.

Furthermore, predictive prescriptions~\citep{bertsimas2020predictive} and, in particular, prescriptive trees~\citep{bertsimas2019optimal,jo2021learning} were introduced, prescribing decisions based on observation. Prescriptive trees are closely related to decision trees, and hence their interpretability is frequently emphasized. The main difference to our approach is that our candidates assigned to the leaves result from an optimization process rather than from classification. In particular, we do not assign prefixed solutions but select suitable ones via optimization. 

Another related research field is contextual optimization \citep{sadana2024survey}, in particular the smart ``predict then optimize'' framework~\citep{elmachtoub2022smart}, which links the prediction of parameter data to the optimization process for which this data is intended. Instead of minimizing the loss of the predicted parameters, the focus is on the loss in cost that results from using predicted parameters compared to true parameters. There have also been advancements in making the predict-and-optimize framework more explainable~\citep{blanquero2023explainable}. Furthermore, the integrated use of decision trees was investigated to increase the interpretability of the approach \citep{elmachtoub2020decision}. Our approach, on the other hand, merges prediction and optimization into a single mechanism, mapping features directly to solutions rather than to parameters that are then used to find an optimal solution. In particular, we do not predict parameters at all; instead, we propose (meta-)solutions.

\subsection{Contribution and Structure}

In this paper, we address the research question of how to obtain interpretable surrogates for optimization and introduce a novel feature-based framework for deriving such optimization rules. While the solution-based approach seeks to find surrogates that map cost parameters to solutions, we instead propose mapping instance \textit{features} to solution \textit{features}. An instance feature can include not only parameters of the given optimization problem but also contextual data. Given solution features can be regarded as a meta-solution, i.e., a representative for a set of feasible solutions for the optimization problem. We therefore extend the flexibility of the setting, thus resulting in better solutions for the decision-maker. We also give insights into the question of how high the cost of enforcing interpretability in optimization is through a computational study.

Our feature-based framework brings several further impacts:
(i) It is easier to explain a meta-solution as we only need to describe the key features corresponding to the solution set instead of an individual solution. Furthermore, features are of a smaller dimension and thus are easier to describe. This will strengthen the acceptance of managerial decisions.
(ii) Using general instance features, rather than only working with coefficients in the objective function, expands the range of problems the framework can cope with. In particular, by using instance and solution features, the dimension of the underlying problem no longer needs to be fixed. The domain now can depend on the instance features, and hence, the framework can be applied to more general settings.
(iii) By providing meta-solutions, we assume that a user is able to identify the specifics of the preferred solution, giving credit to his or her domain knowledge. Additionally, but somewhat contrarily, the optimization rule can be provided to a stakeholder as an argument for why the current solution was implemented.
(iv) From a managerial perspective, providing a meta-solution rather than a specific solution allows the planner greater flexibility in demonstrating that the implemented solution aligns with pre-established rules. However, this approach somewhat calls for ensuring that meta-solutions remain as general as possible to accommodate various interpretations and implementations.

The remainder of this paper is structured as follows. In Section~\ref{sec:framework}, we formally define our framework for feature-based interpretable surrogates and outline how this framework can be applied in the context of different example problems. We then derive mixed-integer programming models in Section~\ref{sec:models}, focusing on the knapsack problem and the shortest path problem, where we also provide hardness proofs. In Section~\ref{sec:heuristics}, we introduce heuristics for our framework. Our methods are applied in Section~\ref{sec:experiments}, where we present numerical experiments involving artificial and real-world data. Finally, Section~\ref{sec::concl} summarizes our paper and provides further research questions.

Throughout the paper, we focus on univariate, binary decision trees of fixed depth to ensure comprehensibility. Hence, by performing easy queries on the instance features to obtain a meta-solution, the optimization process becomes interpretable. Other interpretable mappings from instance features to solution features could be used instead. We use the notation $[n]:=\{1,\ldots,n\}$ to denote index sets and write vectors in bold font. Furthermore, we use ``$\min$'' and ``$\max$'' in optimization problems rather than ``$\inf$'' and ``$\sup$'' to denote the direction of optimization and assume that problems are sufficiently well-posed to allow for the existence of the minimum or maximum, respectively.

\section{Feature-Based Framework for Finding Interpretable Surrogates}
\label{sec:framework}

\subsection{The Framework}\label{subs::framework::framework}

Consider the following nominal optimization problem
\begin{equation*}
    \label{prob:nominal}
    \begin{aligned}
        \min_{\px\in X} c(\px),
    \end{aligned}
\end{equation*}
where $X \subseteq \mathbb{R}^n$ denotes the set of feasible solutions, $n \in \mathbb{N}$ is the number of variables, and $c:X\rightarrow \mathbb{R}$ is an objective function. We assume that an overarching problem domain $\X$ can be modeled, which dictates specific constraints but might lack a fixed dimension. For example, $\X$ may denote the union of all paths or matchings within graphs of various sizes.
In this context, while the fundamental problem setting $\X$ remains the same, several parameters, referred to as \textit{instance features}, are subject to variation. Let $\FIspace\subseteq \mathbb{R}^{F_I}$ be the space of possible realizations of the $F_I\in\mathbb{N}$ instance features. For $\FI \in \FIspace$ the optimization problem is given by
$$\min_{\px \in \X(\FI)} c(\FI,\px),$$
with $\X(\FI)\subseteq \X$ and $c:\FIspace\times\X \rightarrow \mathbb{R}$,
i.e., the instance features can affect the solution space as well as the objective function. We assume a given finite set of $N\in\mathbb{N}$ candidate scenarios also referred to as training scenarios, each represented by its respective instance features $\FI^j\in \FIspace$, $j\in[N]$, with probabilities $\bm{p} \in [0,1]^N$, $\sum_{j\in [N]} p_j=1$. 
We want to find an interpretable optimization rule. 
This is a function that maps a scenario given by its instance features to a meta-solution, i.e.,~a collection of solutions characterized by easily understandable features shared among them. To that end, we assume there exists a set of $F_S\in \mathbb{N}$ \textit{solution features} (stemming from problem-specific knowledge) that describe a solution in an easily comprehensible manner. Let $\FSspace \subseteq \mathbb{R}^{F_S}$ be the solution feature space and for any solution $\px\in \X$, let $\FS(\px) \in \FSspace$ extract solution features from solution $\px$.

The features are selected carefully to i) provide meaningful insights into the structure of the solution, ii) be easily comprehensible for the stakeholder, and iii) allow a practitioner to easily compute a solution when having these features at hand. 
When both the provided features are comprehensible and the circumstances dictating the employment of specific solution features are elucidated, an interpretable surrogate for the optimization process is established. Additionally, in contrast to solution-based interpretable surrogates, by relying on the instance and solution features, the problem dimensions can change, while the provided optimization rule remains intact.

We aim to find an easily comprehensible optimization rule  $a:\FIspace\rightarrow \FSspace$ that maps new instances (given by their instance features) to a set of feasible solutions (given by solution features). In particular, it describes the characteristics of a good solution by providing information on the specifics of each solution feature, depending on the instance features. Solution features do not uniquely describe a solution, but rather a set of solutions. We call $\FS\in\FSspace$ a \textit{meta-solution}, i.e., the solution space specified by the feature characteristics. The used solution in a new instance ought to conform to the specified structure of the meta-solution, while still giving the decision maker some leeway. Hence, even if two different instances call for the same meta-solution, the implemented solutions can differ. We sometimes refer to specific solutions $\bm{x}$ as ``micro''-solutions to emphasize the difference.

As the main goal is to obtain a comprehensible optimization rule, the structure of allowed optimization rules $A\subseteq\{a:\FIspace\rightarrow \FSspace\}$ is crucial. In this paper, we focus on decision trees with a restricted depth, where the branching decisions outline the rule. This way, the number of meta-solutions that can occur is small and the mapping of a scenario to a meta-solution is easily comprehensible. Note, however, that other approaches, e.g., linear optimization rules, can also be employed in this framework.

To formalize our approach, we consider a decision criterion $\mu_{\bm{p}}:\mathbb{R}^N \to \mathbb{R}$ that aggregates a vector of results to a single objective value. Our optimization problem is given by
\[
\min_{a \in A} \mu_{\bm{p}} \left( C(\FI^1,a(\FI^1)),\ldots, C(\FI^N,a(\FI^N)) \right),
\]
where $C:\FIspace \times \FSspace \rightarrow \mathbb{R}$ is a mapping from the instance and solution features to the objective value of one specific solution. In this paper, we consider
$$C(\FI,a(\FI))=\min_{\substack{\bm{x} \in \X(\FI)\\ \FS(\bm{x})=a(\FI)}} c(\FI,\bm{x}),$$
i.e., after the optimization rule has outlined a meta-solution $a(\FI)$, the best solution within the solution space is selected that exhibits these features. 

The general decision criterion $\mu_{\bm{p}}$ allows for the appropriate handling of various settings, e.g., for the Laplace criterion, our optimization problem is given by
$$\min_{a \in A} \sum_{j \in [N]} \min_{\substack{\bm{x} \in \X(\FI^j)\\ \FS(\bm{x})=a(\FI^j)}} c(\FI^j,\bm{x})\ ,$$
whereas for a robust approach, it becomes
$$\min_{a \in A} \max_{j\in [N]} \min_{\substack{\bm{x} \in \X(\FI^j)\\ \FS(\bm{x})=a(\FI^j)}} c(\FI^j,\bm{x})\ .$$

Please note that the framework is designed to be highly general and applicable to a wide range of optimization problems. However, in the following sections, we focus on combinatorial optimization problems.

\subsection{Scope of the Framework}\label{subs::framework::scope}

To demonstrate the scope of our framework and clarify the introduced notation, we explore two example settings in detail. We also outline how our framework can be applied across various optimization domains, highlighting the broader applicability of our approach.

\paragraph{Knapsack.}
Let us first walk through the progressive expansion of the knapsack problem and how our framework can be applied to it. Consider a combinatorial knapsack problem with $n$ items and a budget $C$.
Given a profit vector $\bm{c}$ and a weight vector $\bm{w}$, we can identify an optimal knapsack solution using the following optimization problem:
\[
\max\left\{\sum_{i\in[n]} c_i x_i \text{ s.t. } \sum_{i\in[n]} w_i x_i \leq C,\; \bm{x} \in \{0,1\}^n\right\}.
\]
We now consider a situation with $N$ profit scenarios, that is, we assume that the instance feature vector $\FI$ corresponds to a vector of item profits. Let these be indexed by $j$. To leverage our framework, we assign each item to one of $F_S$ different categories. Let $I_f$ denote the set of items assigned to category $f\in[F_S]$. The meta-solution $\FS$ then consists of specifying the budget for each category, i.e., the $F_S$ solution features specify the amount $C_f$ of the budget spent on category $f\in[F_S]$. We can write the optimization problem for calculating such a budget as
\begin{align*}
\max_{\bm{x},\bm{C}} &\ \sum_{j \in [N]} \bm{c}_j^\top \bm{x}_j \\
\text{s.t.}\ &\sum_{f\in[F_S]} C_f \leq C\\
& \sum_{i \in I_f} w_i x_{ji} \leq C_f && \forall f\in[F_S],\ j\in [N]\\
& \bm{x}_j \in \{0,1\}^n && \forall j\in[N] \\
& C_f \geq 0 && \forall f\in[F_S] .
\end{align*}
In the above setting, we only obtain one single meta-solution (budget capacities for each category) to cover all scenarios. Note that the specific realization of the solution still has some leeway, as the specific selection of items is not prescribed by the meta-solution: it only tells us to use a budget of $C_f$ for items of category $f$ but does not prescribe which exact items to pick.
An even more flexible alternative would be to have multiple possible meta-solutions (though fewer than the number of scenarios) and choose one among them for any given scenario. 
We index this set of meta-solutions by $k\in[K]$.
We can then write the problem as 
\begin{align*}
\max_{\bm{x}, \bm{C}, \bm{\ell}} &\ \sum_{j \in [N]} \bm{c}_j^\top \bm{x}_j \\
\text{s.t.}\ &\sum_{f\in[F_S]} C_f^k \leq C && \forall k\in[K] \\
&\sum_{i \in I_f} w_i x_{ji} \leq C_f^k + M (1-\ell^k_j) && \forall f\in[F_S],\ j \in [N],\ \forall k \in [K]\\
& \sum_{k\in[K]} \ell^k_j = 1 && \forall j \in [N]\\
& \bm{x}_j \in \{0,1\}^n && \forall j\in[N] \\
& \ell^k_j \in \{0,1\} && \forall j\in[N],\ k\in[K] \\
& C^k_f \geq 0 && \forall f\in[F_S], \ k\in[K]\, ,
\end{align*}
where $C^k_f$ is the budget spend on category $f$ in case of meta-solution $k$.
In the above formulation, the meta-solution can be any arbitrary function depending on the instance features. 
To enhance interpretability, we require that the mapping of scenarios to meta-solutions, represented by the indicator variables $\ell^k_j$, be defined by a function that is easily comprehensible. One effective approach is to represent this function as a decision tree. Such a formulation can be expressed as
\begin{align*}
\max_{\bm{x}, \bm{C}, \bm{\ell}, \bm{\theta}} &\ \sum_{j \in [N]} \bm{c}_j^\top \bm{x}_j \\
\text{s.t.}\ &\sum_{f\in[F]} C_f^k \leq C && \forall k\in[K] \\
&\sum_{i \in I_f} w_i x_{ji} \leq C_f^k + M (1-\ell^k_j) && \forall f\in[F_S],\ j \in [N],\ \forall k \in [K]\\
& \bm{\ell}_j = f(\bm{c}_j, \bm{\theta}) && \forall j\in[N]\\
& \bm{x}_j \in \{0,1\}^n && \forall j\in[N] \\
& C^k_f \geq 0 && \forall f\in[F_S], \ k\in[K]\, ,
\end{align*}
where the goal is to find a function $f(\bm{c}_j, \bm{\theta})$ that always takes the form of a decision tree. The splits $\bm{\theta}$ must be optimized so that the function maps instance features (the profit vector) to a meta-solution through this decision tree while maximizing the average profit. Note that this formulation is based on the one presented in~\citep{bertsimas2017optimal}, but other formulations exist as well~\citep{aghaei2024strong}.

We can understand the final knapsack problem in terms of the elements introduced in Section~\ref{subs::framework::framework}.
Here, the set \(\mathcal{X}\) corresponds to the budget constraint of the knapsack problem, which in this example is not affected by the instance features. However, the constraints forming the set are modified to model the meta-solutions.
Each instance feature \(\bm{\phi}_I^j\) reflects a different cost vector for the objective. 
The knapsack problem, as modeled, is equivalent to assigning equal probabilities to all scenarios. 
The solution features \(\bm{\phi}_S\), thus correspond to a meta-solution, which is equivalent to a capacity vector \((C_1^k, \dots, C_{F_S}^k)\) with $\Phi_S\subseteq\mathbb{R}^{F_S}$.
Given a meta-solution and instance features, the function \(C(\FI,\bm{\phi}_S)\) from our framework calculates the optimal objective value obtained from selecting the best singular solution obeying the meta-solution on that instance.
The set \(A\) is the space of all decision trees that map instance features to a meta-solution. 
We parameterize this set by $\bm{\theta}$.
Finally, \(\mu_{\bm{p}}\) represents how the objective values of different instances are aggregated, which is by summing them together, which is equivalent to the Laplace criteria.

\paragraph{Shortest path.}

Assume we are given a graph $G=(V, E)$ representing a road network and nodes $s,t \in V$ representing a starting point and a destination. Furthermore, let us assume that traffic follows a pattern as shown in Figure~\ref{fig:traffic_scenarios}. We assume that there are four main traffic scenarios that can occur at different times of the day throughout the week. We want to provide a comprehensible optimization rule for selecting a route, depending on the time of day and the day of the week.
\begin{figure}[htb]
\centering
         \includegraphics[scale=.4]{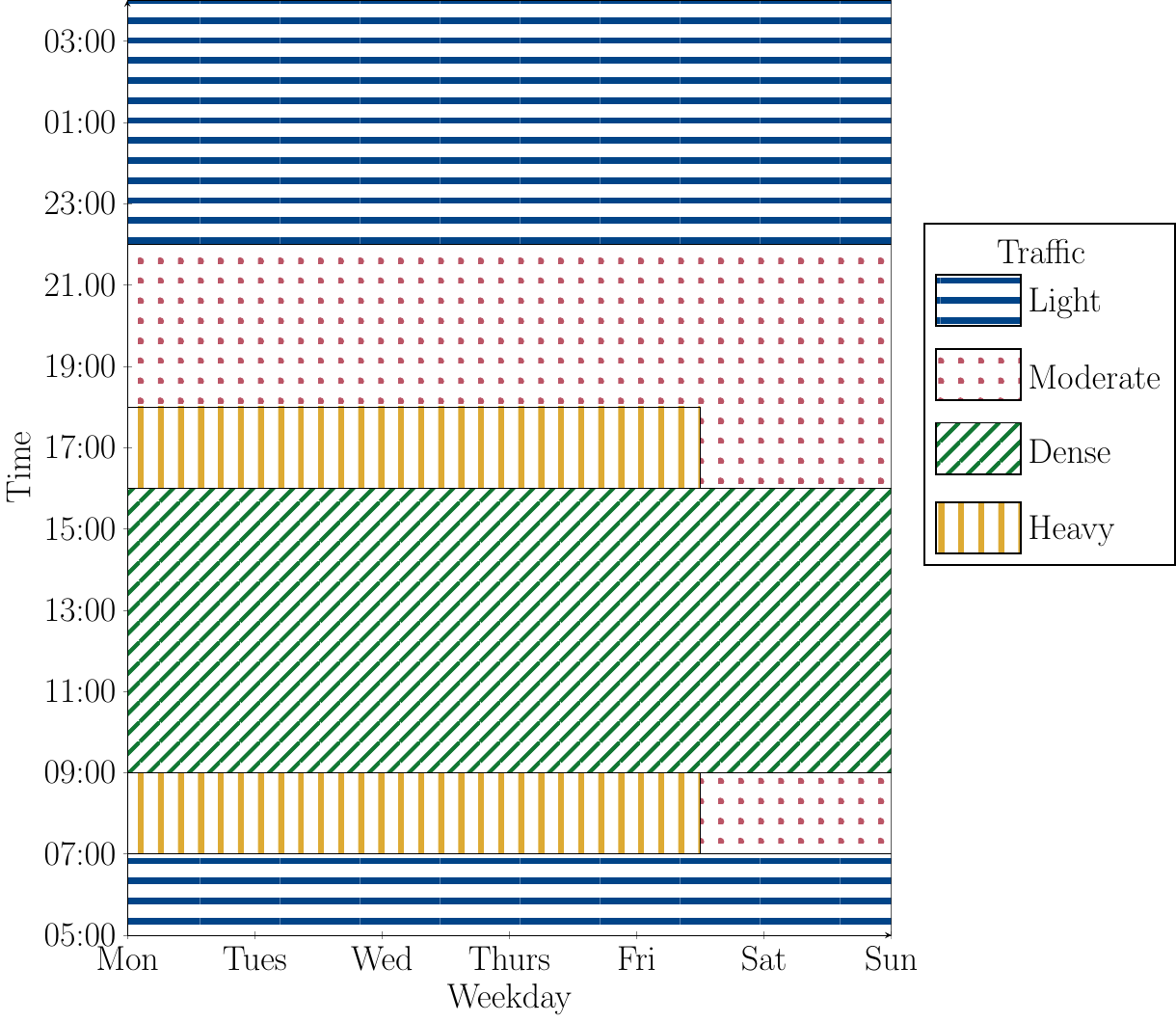}
         \caption{Traffic scenarios assumed in the shortest path example. The day of the week is depicted on the horizontal and the time of day on the vertical axis.\label{fig:traffic_scenarios}}
\end{figure}

To make our example more specific, we are using the graph from~\citet{goerigk_2025_15267992}, representing Chicago's street network, which we have divided into eleven districts (see gray boxes in Figure~\ref{fig::maps_scenarios}).
\begin{figure}
    \centering
    \begin{subfigure}[t]{0.24\textwidth}
         \centering
         \includegraphics[width=\textwidth]{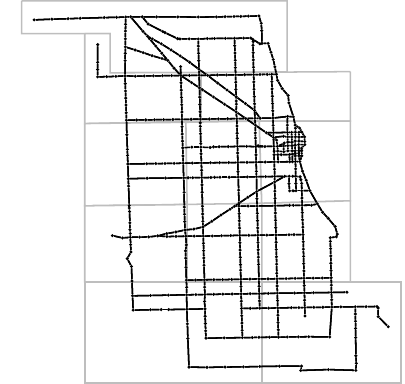}
         \caption{Light traffic.}
         \label{fig::training_data::blue}
    \end{subfigure}
    \begin{subfigure}[t]{0.24\textwidth}
         \centering
         \includegraphics[width=\textwidth]{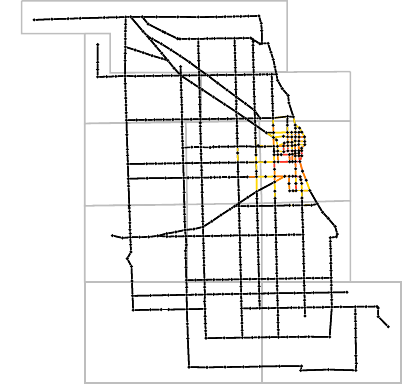}
         \caption{Moderate traffic.}
         \label{fig::training_data::red}
    \end{subfigure}   
    \begin{subfigure}[t]{0.24\textwidth}
         \centering
         \includegraphics[width=\textwidth]{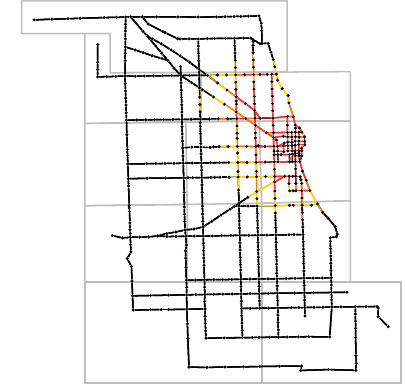}
         \caption{Dense traffic.}
         \label{fig::training_data::green}
    \end{subfigure}
    \begin{subfigure}[t]{0.24\textwidth}
         \centering
         \includegraphics[width=\textwidth]{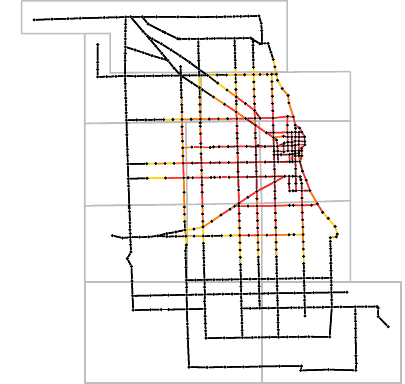}
         \caption{Heavy traffic.}
         \label{fig::training_data::black}
    \end{subfigure}   
    \caption{Traffic scenarios in the city of Chicago.\label{fig::maps_scenarios}}
    \label{fig::training_data}
\end{figure}

The edges are colored if the travel speed on them is slower than usual.  
The colors range from yellow (slightly reduced) to dark red (significantly slower) and provide information about the intensity of the speed reduction. 

For the task of traveling from the northwest to the southeast of the city, let us first consider solutions given by the solution-based interpretable optimization rules.
\begin{figure}
    \centering
    \begin{subfigure}[t]{0.4\textwidth}
         \centering
         \includegraphics[width=\textwidth]{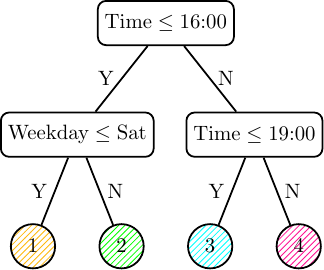}
         \caption{Decision Tree.}
         \label{fig::tree_mgmh::tree}
    \end{subfigure}
    \begin{subfigure}[t]{0.4\textwidth}
         \centering
         \includegraphics[width=\textwidth]{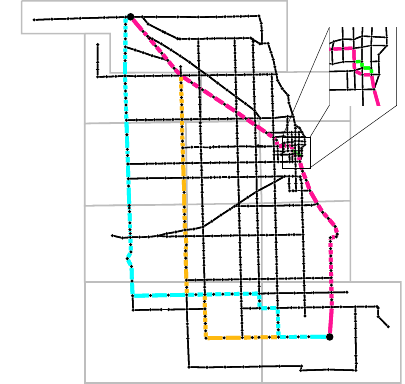}
         \caption{Respective $s$-$t$-paths.}
         \label{fig::tree_mgmh::sols}
    \end{subfigure}   
    \caption{Optimization rule obtained from using the solution-based framework for interpretable surrogates.}
    \label{fig::tree_mgmh}
\end{figure}
Figure \ref{fig::tree_mgmh} illustrates them through a binary decision tree with a maximum depth of two. In each of the four scenarios identified by the decision tree's splits, a specific $s$-$t$ path must be determined.

However, this approach may be too restrictive and diverges from how directions are typically explained in real life, where one would emphasize prominent landmarks, districts, streets, or intersections along the route, rather than detailing every road segment to be traveled. Therefore, we want to provide an optimization rule that uses solution features. In Figure \ref{fig::tree_art::tree}, another binary decision tree is provided which functions as an optimization rule. However, instead of providing $s$-$t$ paths in the leaves, we now provide meta-solutions (shown in Figure~\ref{fig::meta_paths}), i.e., we provide features that any solution should possess in the given scenario. In this specific case, we select as solution features the sequence of city districts that one should traverse to reach the destination.
\begin{figure}[htb]
         \centering
         \includegraphics[scale=1]{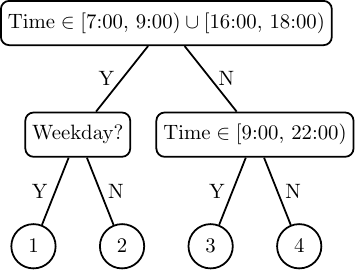}
         \caption{Decision tree. \label{fig::tree_art::tree}}

\end{figure}    
\begin{figure}[htb]
    \centering
    \begin{subfigure}[t]{0.3\textwidth}
         \centering
         \includegraphics[width=\textwidth]{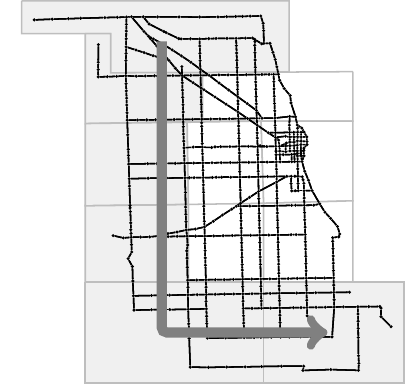}
         \caption{Meta-solution 1.}
         \label{fig::tree_art::s1}
    \end{subfigure}
    \begin{subfigure}[t]{0.3\textwidth}
         \centering
         \includegraphics[width=\textwidth]{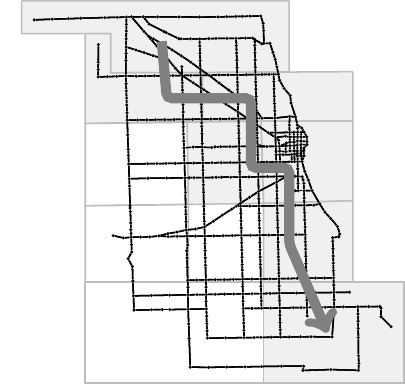}
         \caption{Meta-solution 2 and 4.}
         \label{fig::tree_art::s2}
    \end{subfigure}   
    \begin{subfigure}[t]{0.3\textwidth}
         \centering
         \includegraphics[width=\textwidth]{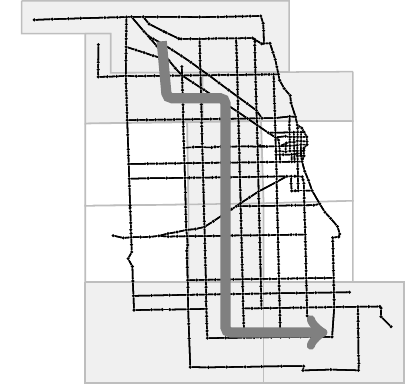}
         \caption{Meta-solution 3.}
         \label{fig::tree_art::s4}
    \end{subfigure}   
    \caption{Meta-solutions for the new framework. Districts shaded in gray have to be visited as per the order indicated by the bold gray arrows. Unshaded districts may not be used.\label{fig::meta_paths}}

\end{figure}
This way, we obtain a more comprehensible specification (specify a few features, rather than a lot of edges) and give the user more leeway in the implementation. In particular, any solution that follows the specified path through the districts is suitable. 

This example also highlights the importance of adequately choosing solution features: E.g., when arriving as a tourist at O'Hare International Airport and searching for the most convenient way to your hotel, a specific route would probably be more helpful than providing district names the tourist is unfamiliar with. In contrast, a courier likely has this kind of knowledge and providing landmarks would suffice in order to guide the route. Hence, it is evident that the solution features have to be carefully matched to the capabilities and expert knowledge of the user.

\paragraph{Other domains.}
We now exemplarily outline how our framework can be applied to various other optimization settings. Given that we have already established the benefits of using comprehensible optimization rules, we briefly discuss how to select instance features (which serve as input for the optimization rule) and solution features (which determine the meta-solution output by the rule).

In a lot-sizing setting, relevant instance features might include product demand, stock levels, production costs, backorder costs, and inventory costs. For a new scenario, the optimization rule could specify solution features such as whether the entire demand is produced in advance (binary indicator), the maximum production amount, the types of products being manufactured, or the overall stock level after production.

In the case of staff scheduling, instance features may include the time or cost of assignments, worker availability, and job demand. Solution features could involve the number of workers from specific groups assigned to each job group, the total number of workers assigned to a job, or binary indicators showing which worker groups are assigned to which job groups.

For the vehicle routing problem, instance features might include demand at nodes, travel costs, and nodes that need to be visited. Solution features could include the number of vehicles assigned to each district, the grouping of specific nodes within the same route, or the (maximum) number of nodes visited by all vehicles or by specific vehicles.

\section{Models}
\label{sec:models}

In this section, we discuss how we can obtain decision tree-based optimization rules using a mixed-integer programming (MIP) formulation and how it can be used to obtain the meta-solutions for knapsack and shortest path problems. Appendix~\ref{appendix:proofs} includes hardness proofs for two special cases of shortest path problems and also highlights the appropriateness of MIP as a modelling tool.

\subsection{Inner Tree Structure}
We start by introducing an MIP formulation for the generation of decision trees which is part of the models in the following sections. The formulation presented is adapted from~\citet{GOERIGK20231312}.
It can be used to generate univariate binary decision trees of fixed depth $Q$.
The presented constraints only govern the inner structure of the tree, i.e., the branching decisions and the assignment of data points to the leaves. 
We also need additional feasibility constraints for the solutions assigned to the tree leaves as well as an objective function that captures the selected decision criteria. Various extensions of this model can be found in~the literature, which are compatible with the formulations presented in this paper.

For simplicity, we present the formulation for symmetric decision trees of depth $Q$.
Such trees require that all queries at any given depth $q \in [Q]$ of the tree be the same. 
For every given data point $j$ and every leaf $k$, a binary variable $\ell_{j}^{k}$ indicates if the data point $j$ reaches leaf $k$ following the tree's queries. 
Each query $q\in[Q]$ is modeled using a binary variable $d_f^q$ which indicates whether feature $f$ is queried at depth $q$.
A continuous variable $b^q$ specifies the threshold for the evaluation. 
It can be bounded by the minimum and the maximum of all feature values, see~\eqref{Eq::TreeStructure::DomainB}. 
The Constraints~\eqref{Eq::TreeStructure::OneSolutionPerScenario} ensure that every data point from the training set is assigned to exactly one leaf of the tree. 
To ensure that every query in the tree chooses a feature, we enforce the Constraints~\eqref{Eq::TreeStructure::Explainable1}.
Constraints~\eqref{Eq::TreeStructure::Threshold1} and~\eqref{Eq::TreeStructure::Threshold2} control the assignment of data points to leaves while taking into account the relevant tree queries. A data point $j$ can reach the left child of query $q$ if and only if the value of the instance feature $f$ of data point $j$ is smaller than or equal to the chosen threshold for this query, see~\eqref{Eq::TreeStructure::Threshold1}. 
Similarly, if the value is greater than the threshold, it can reach the right child, see \eqref{Eq::TreeStructure::Threshold2}. 
To this end, we make use of the index sets $S_q$ to control the assignment of left/right paths to leaves. For a given query node $q$, $S_q \subseteq [K]$ contains the leaves which can be reached by traversing to the right child of $q$. E.g. for $Q := 2$ and therefore $K=4$ we have $S_1 = \{3,4\}, S_2 =  \{2,4\}$.

Note that other formulations for decision trees can also be used in the upcoming models as long as they include the variables $\ell^k_j$, which indicate a mapping of a scenario $j$ to a leaf $k$.

\begin{subequations}
\label{Eq::TreeStructure}
\begin{align}
T(\FI^1,\ldots,\FI^N)= \Big\{ & \bm{\ell}\in\{0,1\}^{K\times N} \text{ s.t. } \exists \bm{d},\bm{b}: \\
&\sum_{k \in [K]} \ell^k_j=1 && \forall j \in [N]
&\label{Eq::TreeStructure::OneSolutionPerScenario}\\
&\sum_{f\in [F_I]} d_f^q = 1 &&\forall q \in[Q] &\label{Eq::TreeStructure::Explainable1}\\
&  \sum_{f\in[F_I]} \phi^{jf}_I d_f^q \le b^q 
+  M\sum_{k\in S_q}\ell^k_j && \forall q \in[Q],\ j\in[N] & \label{Eq::TreeStructure::Threshold1}\\
& b^q+\epsilon -M(1-\sum_{k\in S_q} \ell^k_j)  \le \sum_{f\in[F_I]} \phi^{jf}_I d_f^q && \forall q\in[Q],\ j\in[N] & \label{Eq::TreeStructure::Threshold2}\\
&b^q\in [\min_{j\in[N], f\in[F_I]} \phi^{jf}_I-\epsilon,\max_{j\in[N], f\in[F_I]} \phi^{jf}_I] && \forall q\in [Q] & \label{Eq::TreeStructure::DomainB}\\
&d_f^q \in \{0,1\}&&\forall f \in [F_I],\ q\in[Q]& \label{Eq::TreeStructure::Explainable2}
& \Big\}.
\end{align}
\end{subequations}

\subsection{Interpretable Surrogates for the Knapsack Problem}\label{subs::mod::ks}

We return to the knapsack problem from Section~\ref{subs::framework::scope}.
In this setting, the set of items can be categorized into several subsets. For example, when deciding about in-house budget allocation, the relevant projects can be grouped by their purpose\dash likewise to \citet{chao2008theoretical}. There could be a set of projects affiliated with marketing, production, or research and development. We assume their investment costs to be known, but their pay-off are in the future and hence uncertain. Let us say the optimization rule $a$ takes in the features of the instance (e.g., cost, expected pay-off, demand, macroeconomic factors) and via a small decision tree yields the meta-solutions, i.e., it specifies the budgets for the different categories. For example, in times of high demand, the proportion of budget allocated to production-affiliated projects can be increased. In contrast, whilst facing a low demand, it is desirable to shift a higher share of the total budget towards marketing efforts. Whilst these decision rules are created in agreement with the strategic management of the company, choosing the specific projects to execute is the responsibility of an operational-level manager facing the observed realization.

We assume that the knapsack capacity $C\in\mathbb{R}$ as well as the item weights $w_i \in \mathbb{R}$ are fixed.
Using the Laplace decision criterion and assuming a set of $N$ historic observations, this can be achieved via solving the following optimization problem.

\begin{subequations}
\label{Eq::Knapsack}
\begin{align}
\max \ & \color{black} \sum_{j\in[N]} \color{black} \bm{c}_j^\top \bm{x}_j\\
\textnormal{s.t.}\ & \bm{\ell} \in T(\FI^1,\ldots,\FI^N) \label{Eq::Knapsack::Tree}\\
& \sum_{i\in I_f} w_i x_{ji} \leq C^k_f +M(1- \ell_j^k)&& \forall f \in [F_I],\ k \in [K],\ j \in[N]\label{Eq::Knapsack::xInXf}\\
& \sum_{f \in F} C^k_f \leq C&& \forall k \in [K]\label{Eq::Knapsack::XfsYieldX}\\
& C^k_f \in \mathbb{R}_+&&\forall k \in [K],\ f \in[F_I]\\
& x_{ji} \in \{0,1\} && \forall j \in[N], i \in [n].
\end{align}
\end{subequations}
In comparison to the model we provided in Section~\ref{subs::framework::scope}, we use Constraint~\eqref{Eq::Knapsack::Tree} to model the decision tree mapping from training scenario $j\in[N]$ to meta-solution $k\in[K]$.

\subsection{Interpretable Surrogates for the Shortest Path Problem}\label{subs::mod::sp}
In this section, we describe the formalization of the shortest path approach sketched in Section~\ref{subs::framework::scope}. 
For the shortest path problem, each scenario or instance feature corresponds to a different cost vector. 
The rest of the problem, such as the feasible region and the source and target, remains the same across all the instances. 
As described in Section~\ref{subs::framework::scope}, each node is associated with a district, and the solution features represent a path between the source and target, consisting of a series of districts. Similar to the knapsack formulation, we also use the Laplace criterion in the shortest path objective. 

In contrast to the knapsack model, formalizing the shortest path problem is not so straightforward. The difficulty lies in describing the solution space in the formulation. We want to ensure that every feasible $s$-$t$-path has to follow its given meta-solution. That means that the path must visit the districts (i.e., at least one node of the district) specified in the meta-solution (also referred to as the meta-path) in the specified order. 
We denote the district containing any node $v$ by $\feat(v)$.

For the implementation, it is necessary to limit the number of districts that are part of the sequence representing a meta-solution. This limit is represented by the parameter~$\Delta$. Since any node should not be visited more than once, a trivial upper bound for $\Delta$ is $|V|$. Note that if $\Delta$ is chosen too small, the problem of finding a meta-solution can become infeasible.

Given a graph $G=(E, V)$, we can construct an auxiliary graph $G^\prime = (V^\prime, E^\prime)$ to formulate the meta-solution problem as an IP. 
The set of nodes is given by 
$$V^\prime = \bigcup_{\delta \in [\Delta]} V_\delta,$$
where $V_\delta$ represents the $\delta$th copy of $V$. For $\delta \in [\Delta]$ let $v_\delta$ be the node of $V_\delta$ that is a copy of $v \in V$. The sets $V_\delta$ are, in the following, referred to as \textit{layers}. The set of edges is given by
$$E^\prime = \left(\bigcup_{(u,w) \in E: \feat(u) = \feat(w)}\; \bigcup_{\delta \in [\Delta]} \{(u_\delta, w_\delta) \}\right) \cup \left(\bigcup_{(u,w) \in E: \feat(u) \neq \feat(w)}\; \bigcup_{\delta \in [\Delta-1]} \{(u_\delta, w_{\delta+1})\}\right),$$
i.e., for every edge in $E$ which begins and ends at nodes that belong to the same district, an edge connecting the corresponding nodes in every layer in $G^\prime$ is generated. For every edge in $E$ that begins and ends in nodes that are located in different districts, an edge connecting the corresponding nodes in consecutive layers in $G^\prime$ is created. Recall that a meta-solution refers to the sequence of districts visited in $G$. For the corresponding path in $G^\prime$, only nodes from the same district are visited within each layer. Thus, the meta-solution can be extracted from $G^\prime$, with the solution features representing the districts visited in each layer of $G^\prime$.

 We now consider the problem of finding the shortest path from $s_0$ to some $t_\delta$ in $G^\prime$ with the restriction that for every $v \in V$ we are only allowed to visit at most one of its copies $v_\delta \in V^\prime$. The necessity of explicitly enforcing this condition is proven in Appendix~\ref{appendix:cycles}. Further, we only allow an edge $(u_\delta, w_{\delta + 1}) \in E^\prime$ connecting different layers to be used if the corresponding node $u \in V$ is located in the district which is assigned to layer $\delta$ and the corresponding node $w \in V$  is located in the district which is assigned to layer $\delta +1$. Given a meta-solution and an $s_0$-$t_\delta$-path for some $\delta \in \Delta$, both represented in $G^\prime$, we can obtain an $s$-$t$-path in $G$ which follows the meta-solution.

Figure \ref{fig::layer} shows an example graph $G$ and the corresponding graph $G^\prime$ with $\Delta := 6$. Colors are used to represent the district containing the nodes. Assume that we are given a meta-solution represented by the sequence (\textit{red-green-yellow-green-pink}). Possible resulting $s$-$t$-paths in $G$ are $(s$-$1$-$3$-$4$-$5$-$t)$ and $(s$-$3$-$4$-$5$-$t)$, which correspond to the two paths in $G^\prime$. It can be seen that it is ensured that at least one node in every district listed in the meta-solution has to be visited. The construction also ensures that the order specified by the sequence is adhered to.
Additionally, it is not necessary to assign a district to every layer, as long as the sequence of assigned districts allows for the reachability of a node $t_\delta$  for some $\delta \in [\Delta]$.

\begin{figure}[h!]
    \centering
    \begin{subfigure}[t]{0.36\textwidth}
         \centering
         \includegraphics[width=\textwidth]{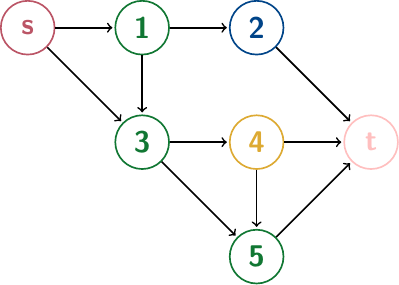}
         \caption{$G$.}
         \label{fig::layer::g}
    \end{subfigure}
    \hspace{.3cm}
    \begin{subfigure}[t]{0.5\textwidth}
         \centering
         \includegraphics[width=\textwidth]{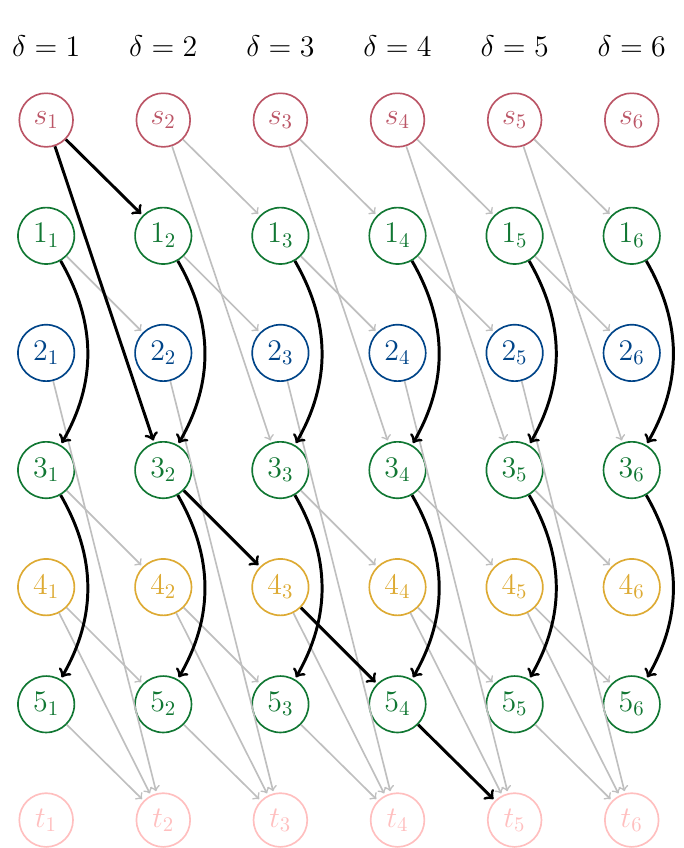}
         \caption{$G^\prime$.}
         \label{fig::layer::gprime}
    \end{subfigure}   
    \caption{$G$ and corresponding $G^\prime$ with an exemplary meta-solution indicated. Edges in $G^\prime$ highlighted in black represent edges that can be used for finding a path from $s_1$ to some $t_\delta$ within the solution space defined by this meta-solution. Grey edges may not be used.}
    \label{fig::layer}
\end{figure}

We use the following sets to formalize this idea in an MIP formulation:
\begin{itemize}
     \item[] \(E^{v\delta}_{-\rightarrow}  = \{ (u_{\delta-1}, v_{\delta}) \in E^\prime \}:\) edges coming into node $v$ in layer $\delta$ from layer $\delta-1$,
     \item[] \(E^{v\delta}_{-\downarrow}  = \{ (u_{\delta}, v_{\delta}) \in E^\prime \}:\) edges coming into node $v$ from within layer $\delta$,
     \item[] \(E^{v\delta}_{+}  = \{ (v_{\delta}, w_{\delta}) \in E^\prime \} \cup \{ (v_{\delta}, w_{\delta+1}) \in E^\prime \}:\) edges going out from node $v$ within layer $\delta$ and to layer $\delta+1$,
     \item[] \(E_{\rightarrow} = \cup_{v \in V} \cup_{\delta \in \Delta} E_{-\rightarrow}^{v\delta}:\) the set of all edges coming into any layer from previous layers.
\end{itemize}

For modeling, we consider edges as belonging to particular layers. 
Specifically, edges within a layer  $(u_\delta, w_\delta)$ and edges which leave it ($u_\delta, w_{\delta+1}$) are considered to belong to the layer $\delta$. For the calculation of the objective value and to ensure that every meta-solution offers the possibility of finding an $s$-$t$-path, we calculate a solution $\bm{x}^\delta_j$ for every scenario $j$. 
The binary variable $x_{ej}^\delta$ is equal to \(1\) if the edge $e \in \{ (u_\delta, w_\delta) \cup (u_\delta, w_{\delta +1}) \in E^\prime : \delta \in [\Delta]\}$ belonging to layer $\delta$ is part of the solution of scenario $j$. 
Constraints~\eqref{Eq::Sp::Source} to \eqref{Eq::Sp::Sink} represent flow conservation constraints, ensuring that $\bm{x}^\delta_j$ represents a valid $s_0$-$t_\delta$-path for some $\delta \in [\Delta]$ in $G^\prime$. Since every node $v\in V$ has multiple representatives in $V^\prime$ it is necessary to ensure that a path found in $G^\prime$ does not contain cycles in $G$. Constraint~\eqref{Eq::Sp::NodeMaxOnce} ensures that at most one representative of any node $v\in V$ can be part of a path in $G^\prime$. 

\small
\begin{subequations}
\label{prob:meta_sp}
\begin{align}
\min \quad & \sum_{j \in [N]} \sum_{\delta \in \Delta} \bm{c}_{j}^\top \bm{x}_{j}^{\delta} \label{eq:ob}\\
\textnormal{s.t.}\ & \bm{\ell} \in T(\FI^1,\ldots,\FI^N) && \label{Eq::Sp::Tree}\\
\cline{1-4}
& \sum_{e \in E_{-\rightarrow}^{s0} \cup E_{-\downarrow}^{s0}} x_{je}^{0} - \sum_{e \in E_{+}^{s0}} x_{je}^{0}  = -1 && \forall j \in [N] \label{Eq::Sp::Source}\\
& \sum_{e \in E_{-\rightarrow}^{v\delta}} x_{je}^{\delta-1} + \sum_{e \in E_{-\downarrow}^{v\delta}} x_{je}^{\delta} - \sum_{e \in E_{+}^{v\delta}} x_{je}^{\delta} = 0 && \forall j \in [N], \delta \in [\Delta], v \in V \setminus \{ s,t \} \label{Eq::Sp::Flow} \\
& \sum_{\delta \in \Delta} \sum_{e \in E_{-\rightarrow}^{t\delta} \cup E_{-\downarrow}^{t\delta}} x_{je}^{\delta} = 1 && \forall j \in [N] \label{Eq::Sp::Sink}\\
\cline{1-4}
& \sum_{\delta \in \Delta} \sum_{e \in E_{-\rightarrow}^{v\delta} \cup E_{-\downarrow}^{v\delta}} x_{je}^{\delta} \leq 1 && \forall j \in [N], v \in V \label{Eq::Sp::NodeMaxOnce} \\
& \sum_{f \in [F_S]} y^{\delta fk} \leq 1 && \forall \delta \in [\Delta], k \in [K]  \label{Eq::Sp::MaxOneF} \\
& x_{je}^{\delta} \leq \frac{1}{2} (y^{\delta \feat(u)k} + y^{(\delta+1)\feat(v) k})+(1-\ell_{j}^{k}) && \forall \delta \in [\Delta],(u,v) \in  E_{\rightarrow}, j \in [N], k \in [K] \label{Eq::Sp::LinkLayer}\\
& x_{j,e}^{\delta-1} \leq (1-\sum_{f \in [F_S]} y^{(\delta+1)fk}) +(1-\ell_{j}^{k}) && \forall \delta \in [\Delta-1], k \in [K], j \in [N], e \in E_{-\rightarrow}^{t\delta} \label{Eq::Sp::SinkLast1} \\
& x_{j,e}^{\delta} \leq (1-\sum_{f \in [F_S]} y^{(\delta+1)fk}) +(1-\ell_{j}^{k}) && \forall \delta \in [\Delta-1], k \in [K], j \in [N], e \in E_{-\downarrow}^{t\delta} \label{Eq::Sp::SinkLast2} \\
\cline{1-4}
& x^{\delta}_{je} \in \{ 0,1\} && \forall e \in E, \delta \in [\Delta], j \in [N] \\
& y^{\delta fk} \in \{ 0,1\} && \forall \delta \in [\Delta], f \in [F_S], k \in [K].
\end{align}
\end{subequations}

It is also essential to couple the district/layer assignment to the edges traversing between the layers. To control these edges, we use the binary variables $y^{\delta fk}$. Edges within one layer are not bounded by them. When set to one, $y^{\delta fk}$ indicates that district $f$ is assigned to layer $\delta$ in leaf $k$. Constraints~\eqref{Eq::Sp::MaxOneF} ensure that at most one feature is assigned to every layer $\delta$ in every leaf $k$. The Constraints~\eqref{Eq::Sp::LinkLayer} restrict the usage of edges between layers for every scenario that is assigned to leaf $k$. An edge $(u_{\delta},w_{\delta+1})$ connecting two consecutive layers can only be used if $y^{\delta fk} =1 : f = \feat(u)$ and $y^{(\delta+1) fk} = 1 : f=\feat(w)$.
Those edges can only be used if the districts assigned to the start and end layers match the features of the start and end nodes of that edge. Constraints~\eqref{Eq::Sp::SinkLast1} and \eqref{Eq::Sp::SinkLast2} ensure that the sink nodes $t_\delta$ can only be visited in layer $\Delta$ or a layer $\delta \in[\Delta-1]$ if the subsequent layer $\delta +1$ has no district assigned, indicating the end of the solution sequence.

Note that if there are leaves that have no assigned scenarios, the formulation does not enforce it to contain a feasible solution. In practice, if such a case happens, the tree can either be pruned in post-processing, or an arbitrary feasible meta-solution can be placed in this leaf.

\normalsize

\section{Heuristics}
\label{sec:heuristics}

While the integer programs constructed in the previous sections can be solved using off-the-shelf solvers, they have limited scalability. The primary alternative to address this issue is to use heuristics to solve the problems.  In this section, we discuss different heuristics for that purpose.
The performance of these heuristics is addressed as a part of the numerical experiments. 

\subsection{Learning Heuristic}\label{subs::heu::learn}
In this heuristic, we first generate training data by solving the underlying optimization problem for each of the $N$ available candidate scenarios. 
We then extract the meta-solutions corresponding to these $N$ optimal solutions. Of these, at most $N$ different meta-solutions, we choose the $K$ best, which are then used to find a classification tree. Algorithm~\ref{alg:sklearn_heuristic} describes the heuristic in detail for minimization problems.
 
\begin{algorithm}
\begin{algorithmic}[1]
\ForAll{$j \in [N]$} \label{heu:l1}
\State $\bm{x}^j \gets \min_{\bm{x} \in \mathcal{X}(\FI^j)} c(\FI^j,\bm{x})$ \label{heu:l2}
\ForAll{$i \in [N]$} \label{heu:l3}
\State $C^*_{ij} \gets \min_{\bm{x} \in \mathcal{X}(\FI^i), \ \FS(\bm{x}) = \FS(\bm{x}^j)} c(\FI^i,\bm{x})$ \label{heu:l4}
\EndFor
\EndFor
\State $(\bm{\mu}, \bm{\lambda}) \gets$ solve Problem~\eqref{Eq::SKIP} with $C^*, K$ \label{heu:l5}
\ForAll{$j \in [N]$} \label{heu:l6}
\State $\ell_j \gets $ index $ i \in [N] \ s.t. \ \mu_{ij} = 1$ \label{heu:l7}
\EndFor
\State \Return classification tree on $\{ (\FI^j, \ell_j) : j \in [N]\}$ \label{heu:l8}
\end{algorithmic}
\caption{\label{alg:sklearn_heuristic} Heuristic algorithm to learn and provide the best meta-solutions.}
\end{algorithm}

More specifically, in Algorithm~\ref{alg:sklearn_heuristic}, we first solve the nominal problem for each scenario $j$ (line~\ref{heu:l2}) to find solution $\bm{x}^j$. Extracting its features, there is a corresponding meta-solution. We then evaluate this meta-solution for each scenario to find its objective value (line~\ref{heu:l4}). If a problem is infeasible, we set $C^*_{ij}=\infty$. Note that we may obtain the same solution $\bm{x}^j$ for multiple scenarios $j\in[N]$, in which case some of the evaluation loops in line~\ref{heu:l3}-\ref{heu:l4} can be skipped.

We then solve an integer program to choose $K$ of the meta-solutions obtained this way. This program is as follows:
\begin{equation}
\begin{aligned}
\text{min} & \sum_{i \in [N]} \sum_{j \in [N]} C_{ij}^* \mu_{ij} &\\
\text{s.t.}\ & \sum_{i \in [N]} \mu_{ij} = 1 && \forall j \in [N]\\
& \mu_{ij} \leq \lambda_i               && \forall i, j \in [N] & \\
& \sum_{i \in [N]} \lambda_i \leq K && &\\
& \mu_{ij} \in \{ 0,1\}                 && \forall i, j \in [N] &\\
& \lambda_i \in \{ 0,1\}                && \forall i \in [N]. &
\end{aligned}
\label{Eq::SKIP}
\end{equation}
Variables $\lambda_i$ are used to determine which of the $N$ meta-solutions are chosen, while the objective function is used to minimize the cost of assigning these meta-solutions to the $N$ scenarios.

Having obtained a choice of $K$ meta-solutions this way, we still need to construct a decision tree. For this purpose, we identify the best meta-solution per scenario (line~\ref{heu:l7}) and train a standard classification tree with limited depth, e.g., using scikit-learn's implementation of the CART algorithm.

\subsection{M2M Heuristic}\label{subs::heu::mgmh}

A second heuristic that we use is based on the approach presented in \citet{GOERIGK20231312}, where a MIP formulation and a greedy heuristic were presented with the aim of creating a decision tree and identifying solutions (not meta-solutions) for the scenarios assigned to its leaves.
The micro-to-macro (M2M) heuristic first uses the MIP or the heuristic from the solution-based approach to find a decision tree that can select solutions based on scenarios.
It then replaces the solutions at the leaves with the respective meta-solution. For all experiments presented in this paper, the MIP formulation was used to obtain the decision trees.

\subsection{Best Single Solutions}\label{subs::heu::single}
In this section, we evaluate the consequences of using only a single solution to optimize performance over all the scenarios. 
We consider the case of both the best meta-solution and the best micro-solution. 

\paragraph{Best single meta-solution.}
In this case, we solve an optimization problem to identify the best single meta-solution to optimize the performance over all the scenarios using the decision criteria:
$$\min_{\bm{\phi}_S \in \Phi_S}\mu_{\bm{p}} \left( C(\FI^1,\bm{\phi}_S),\ldots, C(\FI^N,\bm{\phi}_S) \right).$$
The resulting optimization problem avoids finding an optimization rule and only provides one meta-solution to be applied in every scenario. Such a meta-solution can be obtained by solving the MIP formulations \ref{Eq::Knapsack} and \ref{prob:meta_sp} for $K=1$, which allows removing the variables and constraints regarding the tree structure as well as index $k$.

\paragraph{Best single micro-solution.}
In this case, we even drop the notion of a meta-solution and are only interested in finding a single solution that optimizes the decision criteria:

$$\min_{\bm{x} \in \X}\mu_{\bm{p}} \left( c(\FI^1,\bm{x}),\ldots, c(\FI^N,\bm{x}) \right).$$
Note that this is only applicable if $\cap_{i\in[N]} \X(\FI^i) \neq \emptyset$, i.e., the intersection of the optimization domains of all scenarios is nonempty.

\section{Experiments}
\label{sec:experiments}

This section examines the performance of our approach through numerical experiments on knapsack and shortest path problems.

\subsection{Setting}\label{sub::ex::setting}

To evaluate the performance of our approach and to analyze the trade-off between interpretability and potential objective value, experiments were conducted for two underlying optimization problems. Results for experiments based on knapsack problems are described in Section~\ref{subs::ex::knap}. Results for shortest path problems are described in Section~\ref{subs::ex::sp}. A summary of the results can be found in Section~\ref{subs::ex::sum}.

The methods used are presented in Table~\ref{tab::methods}. To ensure interpretability, we restricted ourselves to trees with two and four leaves. In addition to the tree-based methods, we also compute the single best meta-solution (referred to as \texttt{META1}) and the single best micro-solution (referred to as \texttt{MICRO1}) as described in Section~\ref{subs::heu::single}.
Since here all scenarios are mapped to the same (meta-)solution, both of them can be considered as being the most interpretable of our methods. In contrast, solving every scenario to optimality (\texttt{OPT}) using a black-box solver instead of any mapping represents a non-interpretable method.

\begin{table}[h]
     \centering
     \begin{tabular}{lll}
     \textbf{Method name} & \textbf{Description} & \textbf{Reference} \\ \hline
     \texttt{MIPx}        & MIP formulation                     &  Section~\ref{subs::mod::ks} \& \ref{subs::mod::sp}                  \\
     \texttt{LHx}         & Learning heuristic                  &  Section~\ref{subs::heu::learn}                  \\ 
     \texttt{M2Mx}        & Micro-to-macro heuristic                     &  Section~\ref{subs::heu::mgmh}                 \\
     \texttt{MICROx}      & Solution-based approach                     &  \cite{GOERIGK20231312}                  \\
     \texttt{META1}       &  Single best meta-solution                    &   Section~\ref{subs::heu::single}                 \\
     \texttt{MICRO1}      &  Single best micro-solution                    &   Section~\ref{subs::heu::single}                 \\
     \texttt{OPT} & Optimal solution for each scenario& 
     \end{tabular}
     \caption{List of methods evaluated. We denote the different models as \texttt{NAMEx} where \texttt{NAME} indicates the method and \texttt{x} indicates the number of leaves of the decision trees.}
     \label{tab::methods}
\end{table}

In addition to the experiments using all methods, we also conducted experiments using larger instances. In the latter, the \texttt{MIPx} methods were not used due to their intractability. To generate trees with two leaves, we used the formulation for symmetric trees~\eqref{Eq::TreeStructure}, and for trees with four leaves, we used a formulation for asymmetric trees since its use has proven to be slightly faster.

For both problem types, the values for the vectors $\bm{c}$, representing the costs of edges respectively items, were generated using a variation of the approach from~\citet{GOERIGK20231312}. Three distributions were generated, which we used as base scenarios. The costs of each edge or item were uniformly sampled from one of these distributions, which was chosen randomly with equal probabilities. To generate test sets, we sampled scenarios in the same way. Each test set contains 100 scenarios. More information about the generation of the problem-specific parameters of the instances is included in the respective subsections.

For a given solution $\bm{x}$, its objective value is scaled in the following figures using the formula
$$Obj^{scaled}(\bm{x}) = \frac{Obj(\bm{x}) - Obj(\bm{x}^{\texttt{MICRO1}})}{Obj(\bm{x}^{\texttt{OPT}})- Obj(\bm{x}^{\texttt{MICRO1}})}.$$
A value of one thus represents the objective that can be achieved if each instance is solved by itself to optimality, and is, therefore, an upper bound for all methods used. A higher value thus indicates, even for minimization problems (i.e., in our experiments shortest path), a better solution.  The gap between the value of the scaled objective and $1.0$ is an approximation of the price paid to obtain an interpretable solution instead of the true optimal solution and thus reflects the \emph{cost of interpretability}. Note that the cost of interpretability arises not only from using a surrogate model but also from factors such as the quality of the implemented solution within a meta-solution and the lack of training data.

To evaluate the generated optimization rules, it is necessary to identify the corresponding leaf for a given scenario and compute a feasible micro-solution based on the meta-solution found in that leaf. All our experiments are based on the assumption that the features have been selected reasonably, which allows a stakeholder to practically work with the given meta-solutions. Therefore, we simulate that the user can find an optimal solution for a combination of meta-solution and instance. The resulting optimization problem of finding a micro-solution for the respective meta-solution was thus always solved to optimality. 

We performed 100 runs for each data point presented in the following. To ensure reproducibility, the seeds for generating the instances were fixed. The running times presented include both the time required to solve and construct the problem. We implemented all methods and generated the data using Python version 3.11. As part of the implementation of the learning heuristic (\texttt{LHx}), we used the function \textit{DecisionTreeClassifier} from the scikit-learn library \citep{scikit-learn} for generating the classification trees. We also used the Networkx library \citep{networkx} for storing and handling graph structures. All (M)IP formulations were solved using gurobipy and Gurobi version 11 \citep{gurobi}. Its internal time limit was set to 900 seconds. The experiments presented in Sections~\ref{subs::ex::knap} to~\ref{subs::ex::sp::sizefeat} were conducted on a cluster of machines with 2.6-3.3 GHz Intel Xeon CPUs with 4 cores and 12 GB of memory. For the examples illustrated in Section~\ref{subs::ex::sp::chicago}, a machine with an Apple M2 Pro chip with ten cores and 16 GB of memory was used. The code and the generated data are available on GitHub.

\subsection{Knapsack Problems}\label{subs::ex::knap}
In this section, we present computational experiments using our framework for the knapsack problem. We place particular focus on experiments where the dimension of solution features $F_S$ was varied. The results of further extensive experiments, where the number of items $n$ and the number of used training scenarios $N$ were varied can be found in the Electronic Companion.

All experiments were conducted using artificially generated data. The weights $w_i$ were uniformly sampled from the interval $[0.1,10]$. For our experiments we considered a fixed budget $C$ which was set to $\frac{1}{2} \sum_{i \in [n]} w_i$. For formulating the meta-solutions, the items were grouped as evenly as possible in $F_S$ sets. As described in Section \ref{subs::mod::ks}, the budgets $C^k_f$, which are chosen as part of a meta-solution, do not necessarily sum up to the budget given as part of the instance $C$. Therefore, those were scaled before evaluation such that $\sum_{f \in [F_S]} C^k_f = C$ for all $k \in [K]$ while retaining their ratios.

For the illustrated experiments, we set the number of training scenarios to ten, the number of items to 16, and varied the number of features. The results are presented in Figures \ref{fig::ks::feat-all-tr}, \ref{fig::ks::feat-all-te} and \ref{fig::ks::feat-all-ti}. These experiments include two noteworthy special cases: Firstly, the case where the number of features is set to one. In this case, every given solution space is equal to the set of feasible solutions of the knapsack problem. Due to our assumption that the user can find an optimal solution for a given solution space, the methods \texttt{M2Mx}, \texttt{LHx} and \texttt{MIPx} result in a normalized objective between zero and one for the training as well as the test data. The second case includes the experiments where the number of features was set equal to the number of items. A given meta-solution, therefore, consists of individual budgets for every single item and consequently represents a micro-solution for the knapsack problem. This is reflected in the results to the extent that \texttt{MIPx}, \texttt{MICROx}, and \texttt{M2Mx} have in this case the same objective on the training data. As expected, the gap between the benchmark and our methods closes with an increasing number of features i.e., a competent user can benefit from a rougher description of the solution space, which results in a higher degree of freedom. On the other hand, for an individual who is less familiar with the problem domain, an increasing degree of freedom (i.e., small value $F_S$) could lead to worse solutions.

\begin{figure}[htbp]
    \centering
    \begin{subfigure}[t]{0.49\textwidth}
         \centering
         \includegraphics[width=1\linewidth]{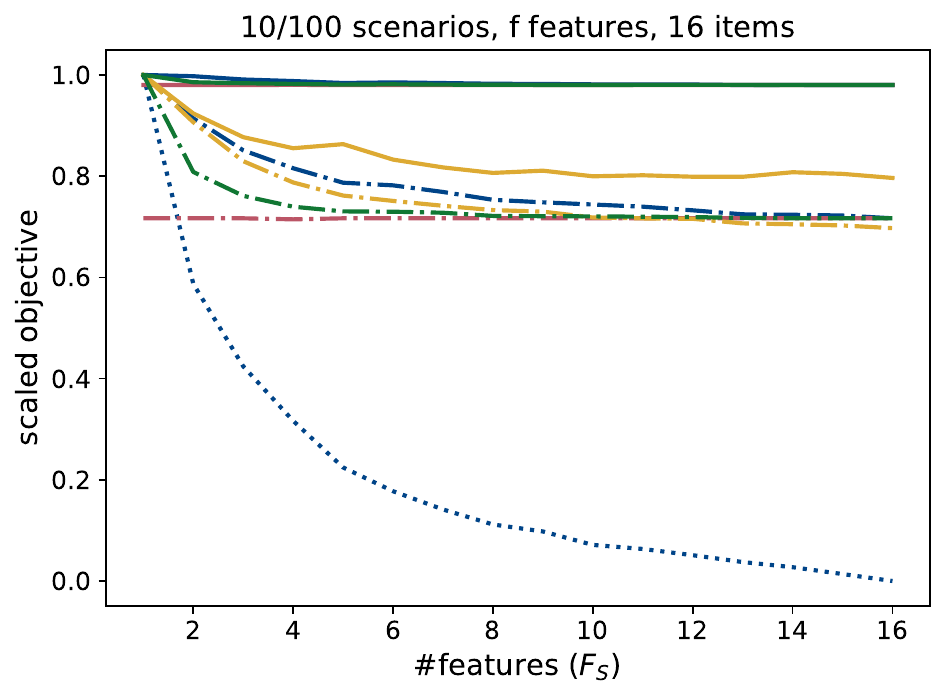}
         \caption{objective training/ \#features.}
         \label{fig::ks::feat-all-tr}
    \end{subfigure}
    \begin{subfigure}[t]{0.49\textwidth}
         \centering
         \includegraphics[width=1\linewidth]{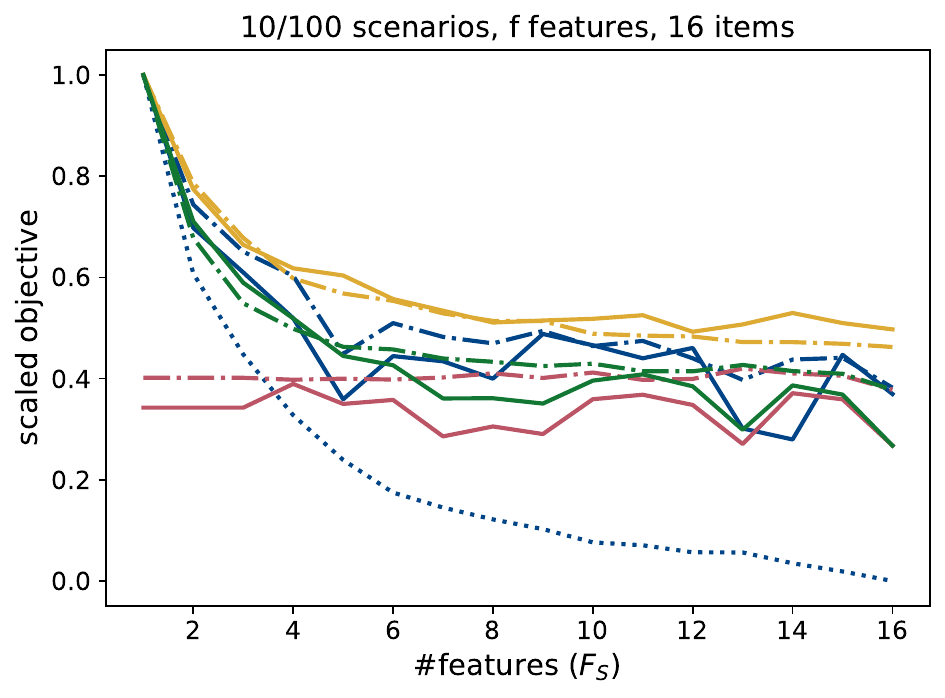}
         \caption{test objective/ \#features.}
         \label{fig::ks::feat-all-te}
    \end{subfigure}
    \begin{subfigure}[t]{0.49\textwidth}
         \centering
         \includegraphics[width=1\linewidth]{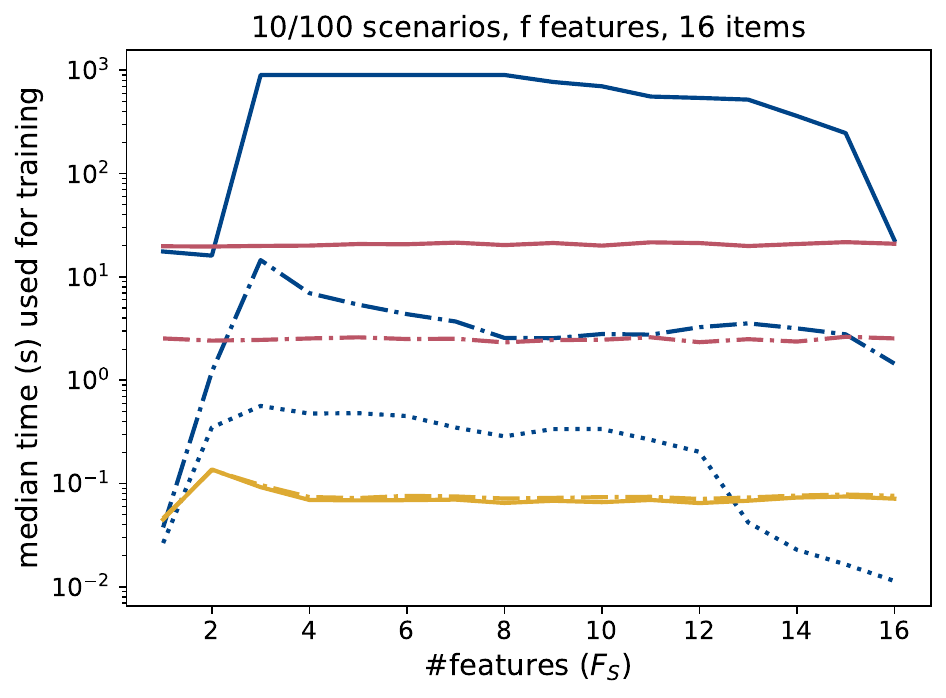}
         \caption{runtime/ \#features.}
         \label{fig::ks::feat-all-ti}
    \end{subfigure}
    \begin{subfigure}[t]{1\textwidth}
         \centering
         \vspace{.4cm}
         \small
\begin{tabular}{clclclcl}
    \raisebox{.3ex}{\includegraphics[scale=.05]{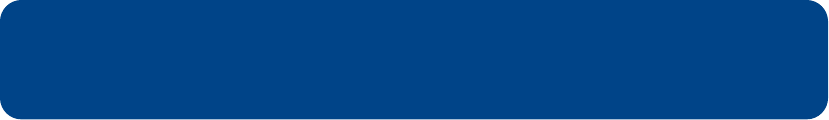}} & \texttt{MIP4} & \raisebox{.3ex}{\includegraphics[scale=.05]{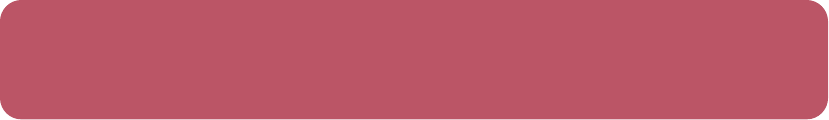}} & \texttt{MICRO4} & \raisebox{.3ex}{\includegraphics[scale=.05]{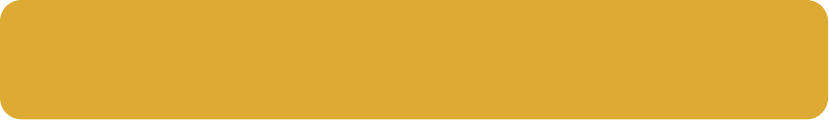}} & \texttt{LH4} & \raisebox{.3ex}{\includegraphics[scale=.05]{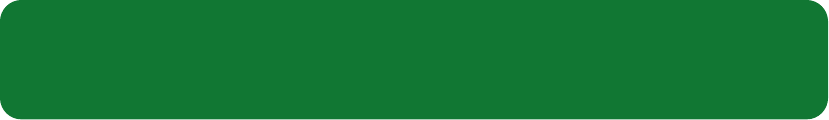}} & \texttt{M2M4}\\
    \raisebox{.3ex}{\includegraphics[scale=.05]{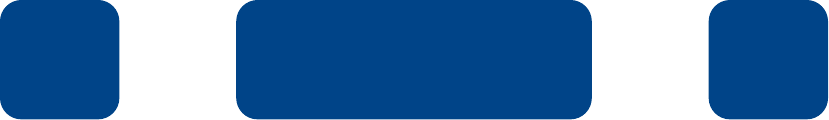}} & \texttt{MIP2} & \raisebox{.3ex}{\includegraphics[scale=.05]{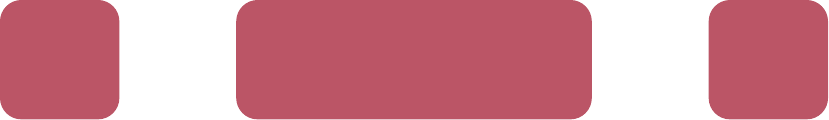}} & \texttt{MICRO2} & \raisebox{.3ex}{\includegraphics[scale=.05]{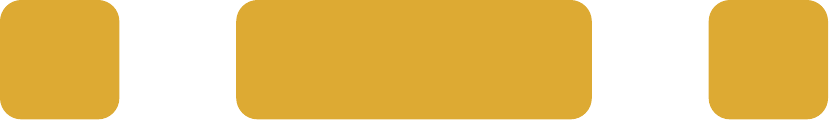}} & \texttt{LH2} & \raisebox{.3ex}{\includegraphics[scale=.05]{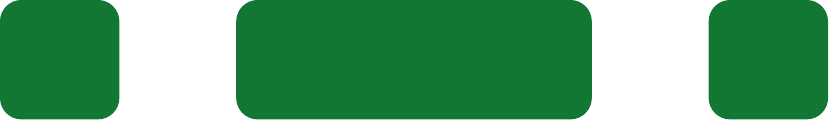}} & \texttt{M2M2}\\
    \raisebox{.3ex}{\includegraphics[scale=.05]{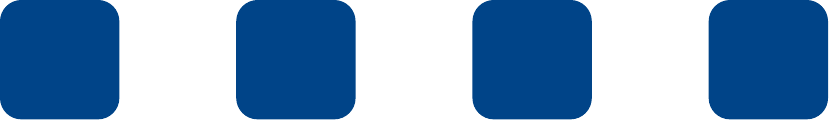}} & \texttt{META1} &
\end{tabular}

    \end{subfigure}
    \caption{Plots for knapsack and synthetic data, \#features on x-axis. A higher scaled objective indicates a better result. The solution-based approach, which serves as a benchmark, is highlighted in red.}
    \label{fig::ks::features}
\end{figure}

In terms of objective performance on the test data, it can be seen that the \texttt{LHx} methods clearly outperform all the other methods. It is especially remarkable that the use of trees with only two leaves generated by \texttt{LH2} nearly always results in better objective values than the use of trees with four leaves made by \texttt{MP4}. Furthermore, it is noticeable that all the MIP-based formulations can be solved particularly quickly for the two extreme cases described.

\subsection{Shortest Path Problems}\label{subs::ex::sp}
In this section, we present results generated by solving shortest-path problems. For artificial graphs, we conducted experiments where the number of scenarios $N$ (Section~\ref{subs::ex::sp::scens}) as well as the dimension of solution features $F_S$ and the number of nodes $n \times n$ (Section~\ref{subs::ex::sp::sizefeat}) were varied.

Those were performed on grid graphs with $n \times n$ nodes, where the node in the southwest corner represents the source and the node in the northeast represents the sink. Every node is connected to its nearest neighbor to its east (if one exists) and to the one in the north (if it exists). The edges are directed from the south to the north and from the west to the east. All graphs used are therefore directed acyclic graphs. The nodes are divided into $F_S$ equally sized square districts, which are used to describe the meta-solutions. In this setting, every meta path from the source $s$ to the sink $t$ has the same length, which is dependent on $F_S$. When solving \texttt{MIP2} and \texttt{MIP4}, the parameter $\Delta$, which regulates the maximum length of a meta-solution, was therefore set to $2\sqrt{F_S} -1$.

In Section~\ref{subs::ex::sp::chicago}, we present extended results for the experiments conducted on Chicago's street network~\citep{chassein2019algorithms} introduced in \ref{subs::framework::scope}. It comprises 1,308 edges and 538 nodes. In these experiments, we used ten scenarios as a training set and 1,000 scenarios as a test set. All scenarios consist of artificially generated edge costs representing four different traffic scenarios presented in Figure~\ref{fig:traffic_scenarios} as well as the time of the day and the day of the week, both encoded as integers.

\subsubsection{Varying Number of Scenarios}\label{subs::ex::sp::scens}

Figures \ref{fig::sp::scens-all-tr} to \ref{fig::sp::scens-heu-ti} illustrate the relationship between objective/time and the number of used training scenarios. The results show that the quality of the solutions on the training data essentially does not increase with more scenarios considered. The solution quality of \texttt{MICRO2} even decreases. It can be observed that\dash as expected\dash the performance on the test data improves as the number of training scenarios increases. While for most methods this is only a slight improvement, the impact on \texttt{MICROx} is quite significant. As with the experiments considering knapsack problems, the required training time shows a dip for \texttt{MICRO4} that we cannot explain. There is an otherwise clear positive correlation between the number of scenarios and the computing time. For all investigated instances, the use of \texttt{LHx} is preferable to that of \texttt{M2Mx}.

\begin{figure}
    \centering
    \begin{subfigure}[t]{0.47\textwidth}
         \centering
         \includegraphics[width=1\linewidth]{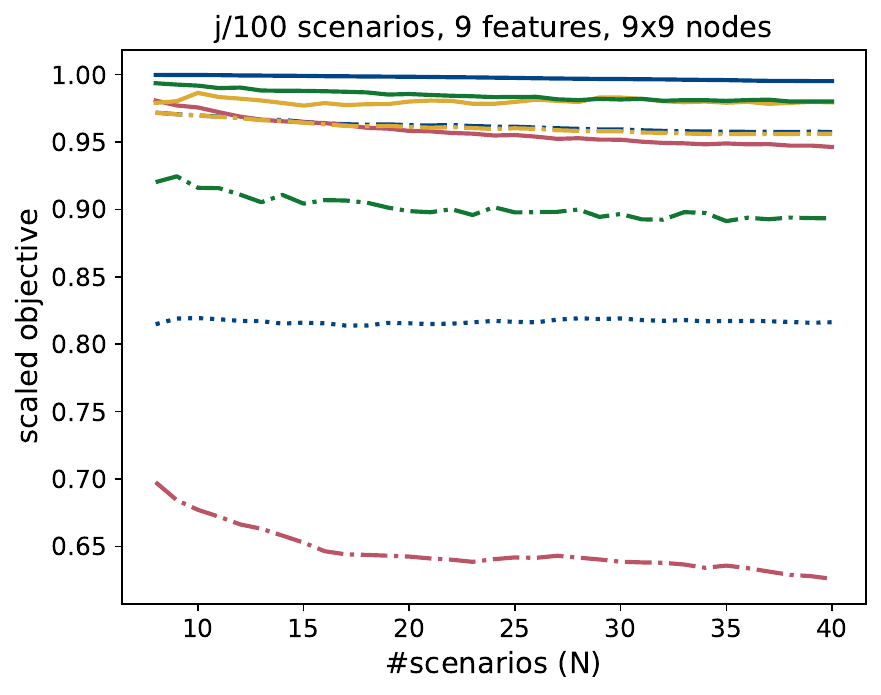}
         \caption{objective training/ \#scenarios.} 
         \label{fig::sp::scens-all-tr}
    \end{subfigure}
    \begin{subfigure}[t]{0.47\textwidth}
         \centering
         \includegraphics[width=1\linewidth]{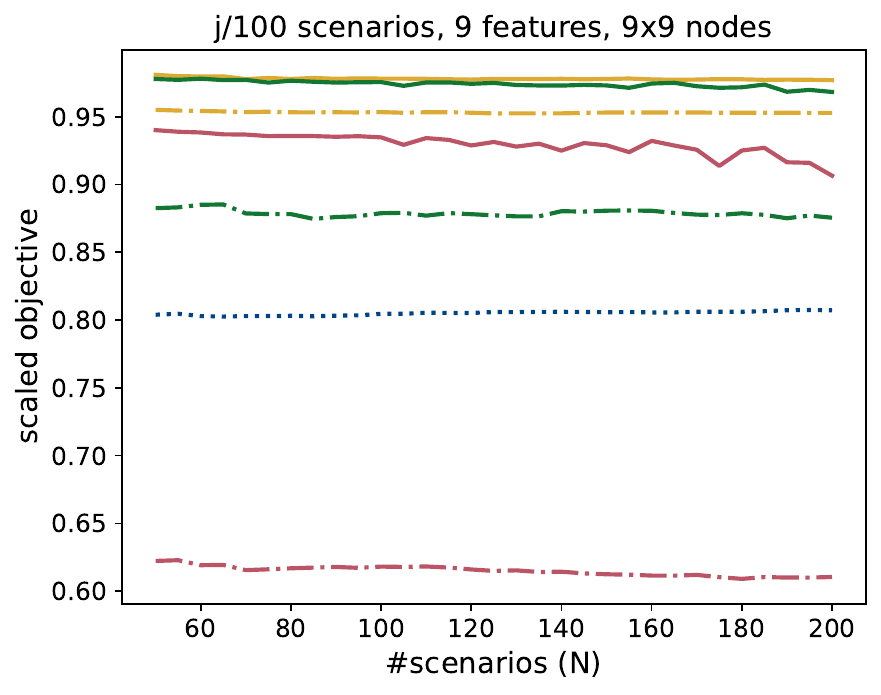}
         \caption{objective training/ \#scenarios.} 
         \label{fig::sp::scens-heu-tr}
    \end{subfigure}   
    \begin{subfigure}[t]{0.47\textwidth}
         \centering
         \includegraphics[width=1\linewidth]{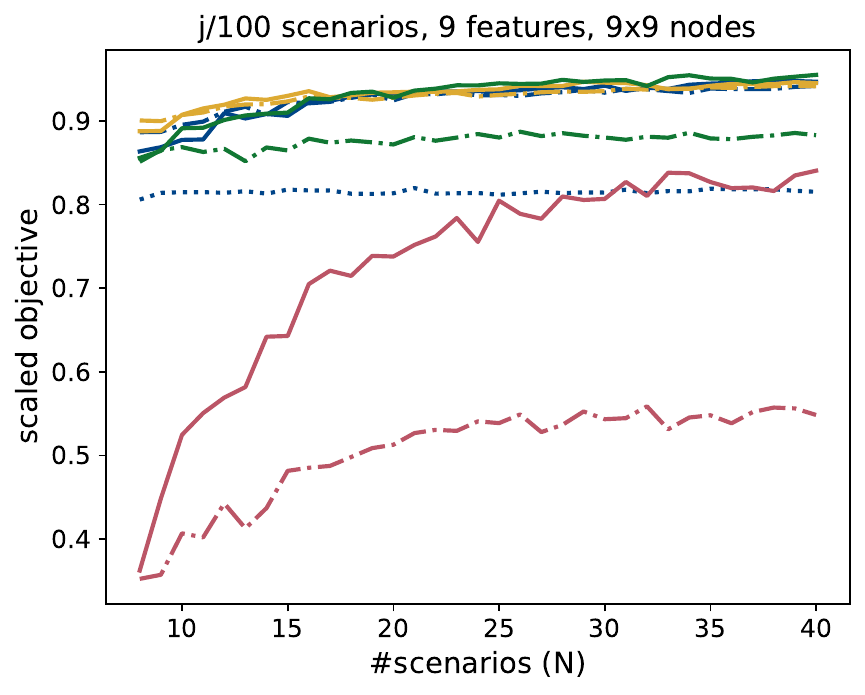}
         \caption{objective test/ \#scenarios.} 
         \label{fig::sp::scens-all-te}
    \end{subfigure}
    \begin{subfigure}[t]{0.47\textwidth}
         \centering
         \includegraphics[width=1\linewidth]{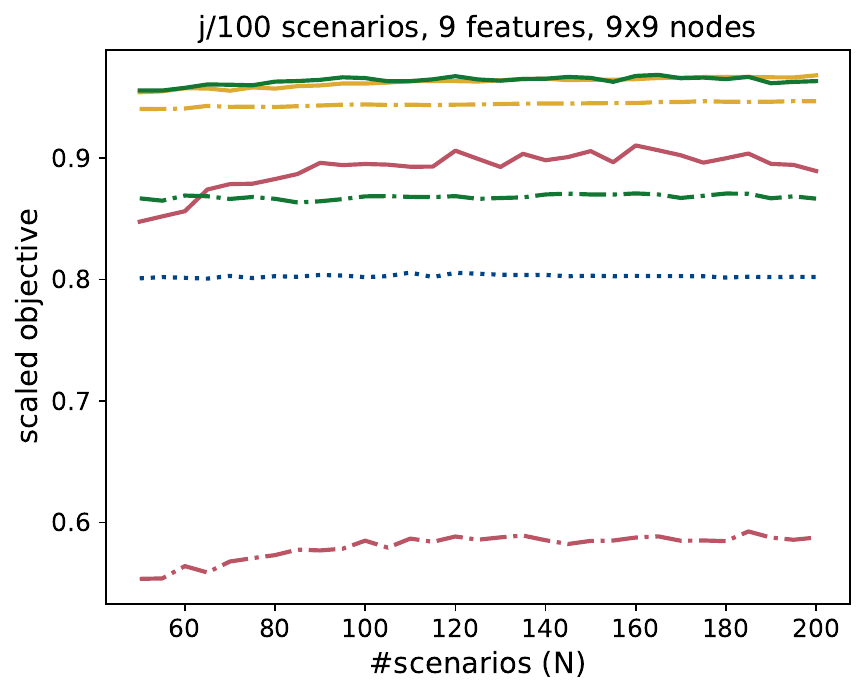}
         \caption{objective test/ \#scenarios.} 
         \label{fig::sp::scens-heu-te}
    \end{subfigure}   
    \begin{subfigure}[t]{0.47\textwidth}
         \centering
         \includegraphics[width=1\linewidth]{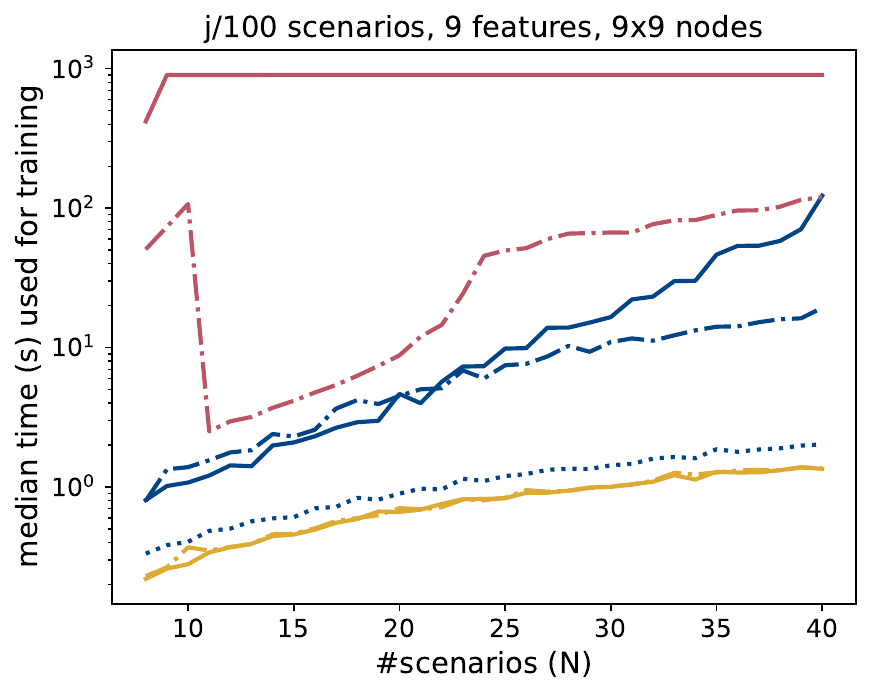}
         \caption{runtime/ \#scenarios.}
         \label{fig::sp::scens-all-ti}
    \end{subfigure}
    \begin{subfigure}[t]{0.47\textwidth}
         \centering
         \includegraphics[width=1\linewidth]{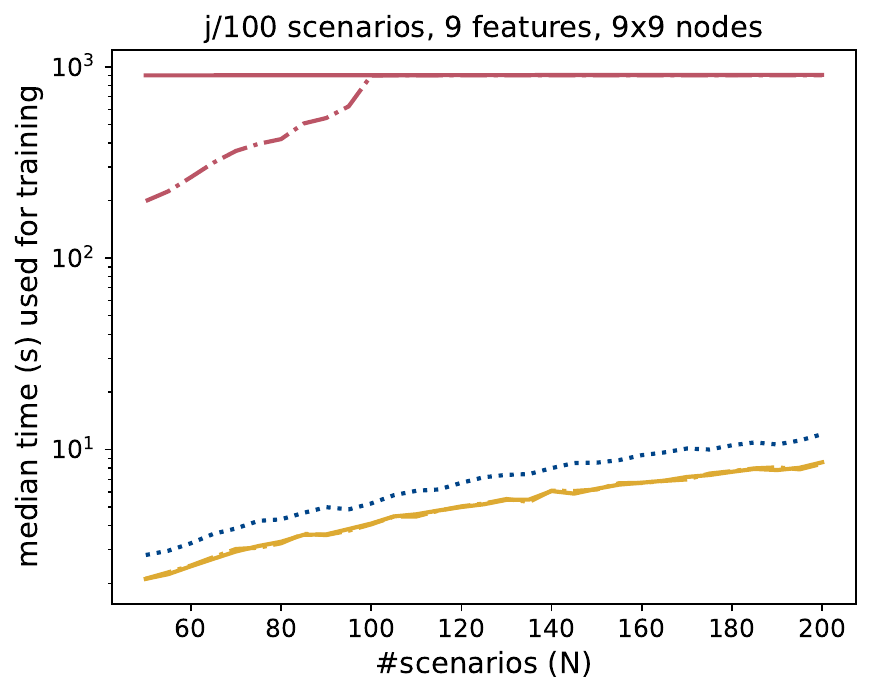}
         \caption{runtime/ \#scenarios.}
         \label{fig::sp::scens-heu-ti}
    \end{subfigure}   
    \begin{subfigure}[t]{1\textwidth}
         \centering
         \vspace{.4cm}
         \small
\begin{tabular}{clclclcl}
    \raisebox{.3ex}{\includegraphics[scale=.05]{legend_blue_full.pdf}} & \texttt{MIP4} & \raisebox{.3ex}{\includegraphics[scale=.05]{legend_red_full.pdf}} & \texttt{MICRO4} & \raisebox{.3ex}{\includegraphics[scale=.05]{legend_yellow_full.pdf}} & \texttt{LH4} & \raisebox{.3ex}{\includegraphics[scale=.05]{legend_green_full.pdf}} & \texttt{M2M4}\\
    \raisebox{.3ex}{\includegraphics[scale=.05]{legend_blue_dasheddotted.pdf}} & \texttt{MIP2} & \raisebox{.3ex}{\includegraphics[scale=.05]{legend_red_dasheddotted.pdf}} & \texttt{MICRO2} & \raisebox{.3ex}{\includegraphics[scale=.05]{legend_yellow_dasheddotted.pdf}} & \texttt{LH2} & \raisebox{.3ex}{\includegraphics[scale=.05]{legend_green_dasheddotted.pdf}} & \texttt{M2M2}\\
    \raisebox{.3ex}{\includegraphics[scale=.05]{legend_blue_dotted.pdf}} & \texttt{META1} &
\end{tabular}

    \end{subfigure}   
    \caption{Plots for shortest path and synthetic data, \#scenarios on x-axis. A higher scaled objective indicates a better result. The solution-based approach, which serves as a benchmark, is highlighted in red.}
    \label{fig::sp::scens}
\end{figure}

\subsubsection{Varying Grid Size and Features}\label{subs::ex::sp::sizefeat}

Due to the type of graphs used, the reasonable choice of parameters $F_S$ and $n$ is limited. The results for experiments where those were varied are presented in Table~\ref{tab::res_sp_grid}. It can be seen that when increasing the size of the graph whilst using the same value for $F_S$, the performance of all methods\dash except \texttt{MICROx}\dash increases. The use of larger graphs furthermore results in larger computation times for all methods used. \texttt{MIPx} can be solved remarkably fast. At the extreme points, when varying $F_S$, as described in detail in Section~\ref{subs::ex::knap}, the known and expected behavior can be recognized. In general, the performance on the training as well as on the test data decreases for all methods except \texttt{MICROx}. Except for the latter, the computation time needed increases with the number of features used. It is worth mentioning that the \texttt{MIPx} methods can be solved significantly faster for $F_S=81$ than \texttt{MICROx}, although both represent the same problem.

\begin{table}[htbp]
    \centering
\begin{tabular}{crrrrrrrrrr}
\multicolumn{1}{l}{} &    &    & \multicolumn{4}{c}{2 leaves}                                  & \multicolumn{4}{c}{4 leaves}                                         \\
\multicolumn{1}{l}{} & \multicolumn{1}{r}{$n$}  & \multicolumn{1}{r}{$F_S$}  & \texttt{MIP}  & \texttt{MICRO}& \texttt{LH}   & \texttt{M2M}  & \texttt{MIP}            & \texttt{MICRO} & \texttt{LH}   & \texttt{M2M}  \\
\hline
\parbox[t]{2mm}{\multirow{5}{*}{\rotatebox[origin=c]{90}{Training}}} & 9  & 1  & \textbf{1.00} & 0.67   & \textbf{1.00} & \textbf{1.00} & \textbf{1.00} & 0.98          & \textbf{1.00} & \textbf{1.00} \\
                     & 9  & 9  & \textbf{0.97} & 0.67   & \textbf{0.97} & 0.92          & \textbf{1.00} & 0.98          & 0.99          & 0.99          \\
                     & 9  & 81 & 0.67          & 0.67   & 0.63          & \textbf{0.68} & \textbf{0.98} & \textbf{0.98} & 0.80          & \textbf{0.98} \\
                     & 6  & 9  & \textbf{0.94} & 0.71   & \textbf{0.94} & 0.88          & \textbf{1.00} & 0.98          & 0.98          & 0.99          \\
                     & 12 & 9  & \textbf{0.96} & 0.49   & \textbf{0.96} & 0.88          & \textbf{1.00} & 0.83          & 0.98          & 0.96          \\
                     \hline
\parbox[t]{2mm}{\multirow{5}{*}{\rotatebox[origin=c]{90}{Test}}} & 9  & 1  & \textbf{1.00} & 0.40   & \textbf{1.00} & \textbf{1.00} & \textbf{1.00} & 0.51          & \textbf{1.00} & \textbf{1.00} \\
                     & 9  & 9  & 0.90          & 0.41   & \textbf{0.91} & 0.87          & 0.88          & 0.52          & \textbf{0.91} & 0.89          \\
                     & 9  & 81 & \textbf{0.43} & 0.39   & 0.42          & 0.39          & \textbf{0.52} & 0.50          & 0.50          & 0.50          \\
                     & 6  & 9  & 0.84          & 0.38   & \textbf{0.85} & 0.77          & 0.82          & 0.49          & \textbf{0.85} & 0.83          \\
                     & 12 & 9  & \textbf{0.90} & 0.16   & \textbf{0.90} & 0.85          & 0.88          & 0.18          & \textbf{0.89} & 0.87          \\
                     \hline
\parbox[t]{2mm}{\multirow{5}{*}{\rotatebox[origin=c]{90}{Time (s)}}} & 9  & 1  & 0.22          & 100.50 & \textbf{0.11} & 100.50        & 0.49          & 900.74        & \textbf{0.11} & 900.74        \\
                     & 9  & 9  & 1.39          & 107.06 & \textbf{0.37} & 107.06        & 1.08          & 900.70        & \textbf{0.28} & 900.70        \\
                     & 9  & 81 & 16.96         & 84.39  & \textbf{1.32} & 84.39         & 19.16         & 900.63        & \textbf{1.25} & 900.63        \\
                     & 6  & 9  & 0.32          & 33.74  & \textbf{0.16} & 33.74         & 0.43          & 448.02        & \textbf{0.14} & 448.02        \\
                     & 12 & 9  & 3.32          & 211.25 & \textbf{0.63} & 211.25        & 1.94          & 901.08        & \textbf{0.59} & 901.08       
\end{tabular}
\caption{Results for varying grid size and features.}
\label{tab::res_sp_grid}
\end{table}

\subsubsection{Chicago Graph}\label{subs::ex::sp::chicago}

For the trees generated using the Chicago graph, we restricted ourselves to trees with a maximum depth of two and the use of the instance (meta) features weekday and time. In Table \ref{tab::chicago2}, the scaled objectives and computational times are presented. Our example tree presented in Figure \ref{fig::tree_art::tree} is referred to as \texttt{ART4}.

\begin{table}[h!tbp]
    \centering
    \begin{tabular}{cccccc}
        & \texttt{MIP4} &\texttt{MICRO4}& \texttt{LH4} & \texttt{M2M4} & \texttt{ART4}\\
        \hline
        Training & \textbf{0.95} & 0.90 & 0.92 & 0.94 & 0.92 \\
        Test & 0.86 & 0.40& 0.82 & 0.78& \textbf{0.89} \\
        Time (s) & 23.22 & 6.22 & 6.08 & 6.22 & --
    \end{tabular}
    \caption{Results for the Chicago instance.}
    \label{tab::chicago2}
\end{table}

It can be seen that, congruent with the previous results, all the methods using meta-solutions achieve better objective values than our benchmark \texttt{MICRO4}. Especially, both of the heuristics lead to good results on the test data. Even though we were using a large graph, the methods \texttt{MIP4} and \texttt{MICRO4} could be solved in an acceptable time.

\subsection{Discussion and Managerial Insights}\label{subs::ex::sum}

Summarizing the presented results, we see that solving \texttt{MIPx} results in genuinely good solutions. It was shown that for most instances, the formulations could not be solved to optimality within the time limit of 900 seconds. Using this method, therefore, takes much more time than solving the nominal problems. However, it should be noted that in real-world applications, the found optimization rule would be used many times, which amortizes the computational effort. Furthermore, feasible solutions were found in every case within the time limit.

Using the learning heuristic (\texttt{LHx}), we generated solutions that were often nearly as good, and in some cases even better, than those produced by \texttt{MIPx}. This is remarkable as the trees could be generated using the heuristic in a fraction of the time limit.
Regarding the objective value, the heuristic \texttt{M2Mx} performed mostly worse than \texttt{MIPx} and \texttt{LHx}. Furthermore, it has been shown not to scale well due to the necessity of solving an MIP formulation as part of the process. This process can be sped up by using heuristics for generating trees with micro solutions assigned to their leaves, e.g., via greedy heuristics. Our results imply \texttt{LHx} to be preferable to \texttt{M2Mx} in general and to \texttt{MIPx} at least for big instances, which should occur in reality in most cases.

Due to the restriction of only considering $K$ solutions and the requirements of the tree structure, none of our methods were able to achieve the objective of the black-box model~(\texttt{OPT}). Nevertheless, in comparison to the interpretable approach from~\citet{GOERIGK20231312} (\texttt{MICROx}), our framework has the potential to produce\dash depending on the value of $F_S$\dash better solutions. In some cases, trees with two leaves and meta-solutions even outperformed trees with four leaves and micro-solutions on the test data. Therefore, we can further close the gap between the performance of interpretable surrogates and black-box models whilst being able to increase interpretability at the same time by the use of smaller trees.

It could also be shown that an increasing ratio between the dimension of features and scenario size leverages the performance of our approach. For example, in the experiments presented in the Electronic Companion, for values of $F_S$ in the interval $[2, 6]$ whilst keeping the instance size fixed, \texttt{MIP4} results in solutions which are 47.57\% better on average compared to \texttt{MICRO4} considering their scaled objective. For values of $F_S$ in $[7, 15]$ a smaller gap of 25.57\% on average\dash but in favor of \texttt{MIP4}\dash can be observed. Without the context of a real-world application, our results suggest therefore to set $F_S$ to a relatively small value while ensuring that the selected features are still interpretable. A small value of $F_S$, though, increases the requirements on the user of the decision tree since the solution space described by the solution features increases. This could lead to a point where, in reality, optimal micro-solutions can not be guaranteed and the quality of solutions decreases. Note that this result does not hold in general. For example, there are cases in which adding meaningless features does not result in a decrease in solution quality and vice versa. 

The ability to work with a given meta-solution is highly dependent on the user and the given use case. We therefore do not consider this in our computational experiments. The performance gain in real-world applications is hence related to a reasonable choice of the solution features and their dimensions. 

Our presented framework is suitable if an interpretable solution process for an optimization problem is needed, and the decision maker who is executing the generated optimization rules has some domain knowledge as well as the authorization to make decisions on their own. If the latter is not the case, we recommend using an interpretable method, where strict solutions are found during the optimization process. Both approaches require historical or at least estimated data for the generation of the optimization rules. Since both use small decision trees to represent the optimization rules, interpretability is ensured which helps to increase user acceptance.

Our approach provides optimization rules that map a given scenario to a space of possible solutions. It therefore allows\dash and requires\dash a final decision, which is the actual solution to implement. This more dynamic approach allows the user to react to unpredictable events and hence generate better solutions. Furthermore, it helps to avoid undesired micromanagement.

The options for generating training data are case-specific. If meaningful historical data is available, its use can be considered, otherwise, data must be estimated. Depending on the size of the problem, \texttt{MIPx} or \texttt{LHx} should be used for the training of the decision tree. For most bigger instances, applying \texttt{LHx} will be beneficial due to a much shorter runtime, whilst for small instances or in use cases where sufficient time is available, \texttt{MIPx} can potentially generate better solutions. The solution features have to be chosen in a way that enables the user to understand the meta-solutions while still providing a sufficient degree of freedom.

\section{Conclusion}
\label{sec::concl}

Interpretable optimization surrogates aim to bridge the gap between optimization theory and practice by accounting for the methods used by decision-makers and stakeholders. By providing solution methods that are transparent and easy to apply, we avoid what is de facto black-box optimization and therefore improve the acceptance of optimization algorithms. As this topic has been introduced only recently, the current toolbox of interpretable methods is still very limited. 

In this paper, we introduce a flexible framework to find the best optimization rule to map instance features to solution features. This means that a post-optimization step becomes necessary to convert the resulting solution features (also referred to as a meta-solution) to a single, feasible solution. In the context of a shortest path problem, for example, this may mean that a given sequence of districts must be converted to a path. The granularity of features should thus be modeled in a way that a decision-maker can perform this step.

We demonstrated how to apply this framework in different settings, with a special focus on knapsack and shortest path problems. We introduced mixed-integer programming formulations as well as heuristics, in particular a learning heuristic that leverages existing fast and reliable methods to construct classification trees. We analyzed the complexity of our framework and showed that simple and natural special cases already turned out to be hard. In extensive computational experiments involving randomly generated data as well as real-world data, we analyzed the performance of our framework and found that the additional flexibility provided by the use of features can greatly improve the quality of the produced optimization rule, thus reducing the cost of interpretability, i.e., the difference to what is achievable by classic, non-interpretable solution methods.
This reduced cost, along with the inherent benefits brought on by interpretability, such as acceptability and responsiveness, makes it worthwhile as a manager to put resources into exploring the benefits of interpretability and to investigate how different stakeholders might react to more comprehensible optimization results.

As interpretability in optimization is a recent field of research, we believe there are many further directions to study. One aspect is to analyze the impact if the decision maker is not able to choose an optimal solution for a given meta-solution and instead relies on suboptimal outcomes. 
In this case, we might define a probability distribution over the set of all solutions that correspond to a meta-solution to obtain an expected performance metric (i.e., a stochastic approach), or even consider a worst-case solution from this set (i.e., a robust approach). Uncertainty may not only be present in the conversion of meta-solutions to solutions but also in the problem data. In this case, it may be beneficial to construct decision trees that take uncertainty in training data into account.

\section{Code and Data Disclosure}\label{sec:Code and Data Disclosure}The code and data to support the numerical experiments in this paper is available on demand.

\newpage
\appendix

\section{The Necessity of Preventing Circles in the Solution of Shortest Path Based Instances}\label{appendix:cycles}

The question of whether it is necessary to actively avoid cycles in the concept presented in Chapter~\ref{subs::mod::sp} arises, since it seems counterintuitive that an optimal solution to a shortest path problem with non-negative edge costs contains cycles when optimizing w.r.t. the Laplace criterion. As the assumption of having to suppress them plays a significant role in the proofs presented in Appendix~\ref{appendix:proofs}, this question has direct and clear effects on the correctness of the hardness results. We therefore present an example illustrating the necessity of prohibiting them.

\begin{figure}[ht]
    \centering
    \includegraphics[width=0.3\linewidth]{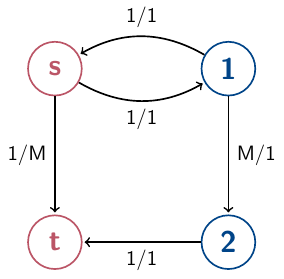}
    \caption{Example for a shortest path instance with two scenarios and two solution features highlighting the need to actively avoid cycles in the solutions.}
    \label{fig:cycles}
\end{figure}

Consider the instance of a shortest path problem depicted in Figure~\ref{fig:cycles}. We are given two cost scenarios which are displayed on the corresponding edges. Let $M$ be a large number. Furthermore, let $f(s)=f(t)=red$ and $f(1)=f(2)=blue$. We now want to determine a meta solution for finding a path from $s$ to $t$ such that the costs with respect to the Laplace criterion are minimized.

First, cycles in the paths are not considered prohibited. In this case, the (unique) optimal meta path is the sequence (\textit{red-blue-red}) which allows the paths (\textit{s-1-s-t}) for the first and (\textit{s-1-2-t}) for the second scenario. This meta solution, therefore, results in an objective value of~$6$. It becomes obvious that this solution exploits the possibility of cycles in the paths. If this meta solution is evaluated in an environment where paths containing cycles are considered infeasible, the optimal path in scenario one becomes (\textit{s-1-2-t}), resulting in an overall objective value of $M+5$.

Indeed, the (also unique) optimal solution for the problem of finding an optimal meta path if cycles are prohibited is given by the sequence ($red$) and the corresponding paths (\textit{s-t}) and (\textit{s-t}). This results in an objective value of only $M+1$.

Note that this example highlights the necessity of actively suppressing cycles in paths when solving the problem of finding the best meta solution, as well as when solving the problem of finding optimal paths for a given meta solution.

\section{Hardness Results}\label{appendix:proofs}

We now turn to the complexity of finding interpretable surrogates for the shortest path problem. As it is possible to solve the nominal shortest path problem in graphs without negative cycles efficiently, the question arises whether polynomial-time algorithms also exist for finding interpretable optimization rules in this case. In the following, we give negative answers to two special cases.

\begin{theorem}\label{th:metahard1}
The following problem is NP-complete: Given a graph $G = (V, E)$ and a meta-path, decide whether there is a path from node $s$ to node $t$ in $G$ that corresponds to the meta-path.
\end{theorem}

\begin{proof}{Proof}
A reduction from the NP-complete Hamiltonian path problem~\citep{garey1979computers} is constructed. Given an undirected graph $G=(V,E)$ with nodes $s,t\in V$, we need to decide whether there is a (simple) $s$-$t$-path visiting all nodes in $V$ exactly once. We first create a graph $\bar{G} = (\bar{V} ,\bar{E})$ in the following way: For every $v \in V$, two nodes $v$ and $v^\prime$ are added to $\bar{G}$ as well as an edge $(v,v^\prime)$. Node $v$ is assigned to district $a$, and node $v^\prime$ to district $b$. For every $\{i,j\} \in E$, the edges $(i^\prime,j)$ and $(j^\prime,i)$ are added to $\bar{G}$.
Figure~\ref{fig:proof1} shows an example of this construction.

\begin{figure}[h!]
\centering
\includegraphics[scale=.8]{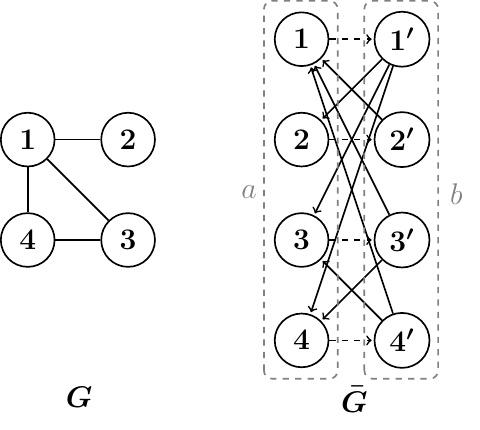}
\caption{Example of construction for the proof of Theorem~\ref{th:metahard1}.}
\label{fig:proof1}
\end{figure}

We define a meta-path of length $2|V|$ consisting of an alternating order of the districts $a$ and $b$, starting with $a$. Let $\bm{x}^\prime$ be a simple path connecting $s$ and $t^\prime$ in $\bar{G}$ corresponding to the meta-path. Since this path has to include $2|V| = |\bar{V}|$ many nodes, every node in $\bar{G}$ is visited exactly once. We can derive a path in $G$ which visits $|V|$ nodes exactly once by only considering the edges of type $(i^\prime,j)$. Similarly, any Hamiltonian path in $G$ beginning in $s$ and ending in $t$ can be extended to a feasible path in $\bar{G}$.

Therefore, there exists a Hamiltonian path in $G$ if and only if the meta-path problem instance is a yes-instance.
\end{proof}

\begin{theorem}\label{th:metahard2}
The following problem is NP-complete: Given a graph $G=(V, E)$ where each node $v\in V$ belongs to a district $f(v)$, and a list of cost scenarios $\bm{c}^1,\ldots,\bm{c}^N\in\{0,1\}^E$. Decide whether there is a meta path such that there exist paths that satisfy the meta path which have a cost of 0 for each scenario. In particular, the problem of finding a single meta-path with minimum costs is not approximable, even in directed acyclic graphs.
\end{theorem}

\begin{proof}{Proof}
Let an instance of the NP-complete 3SAT problem~\citep{garey1979computers} be given. It consists of $n$ Boolean variables $x_1,\ldots,x_n$ and a list of clauses $C_1,\ldots, C_m$, where each clause is a disjunction of three literals. We need to decide if there is a truth assignment of the $n$ variables such that all $m$ clauses are true.

We first show how to construct a graph $G'$ 
with $2n+2$ nodes contained in the set $V'$. For each variable $x_i$, we construct a node $t_i$ and a node $f_i$, which represent the states of ``true'' and ``false'', respectively. We further introduce nodes $s$ and $t$ to be used as the source and target nodes of the meta-path. The set of edges is given by
\begin{align*}
E' = & \Big\{ (a,b) : a\in\{t_i,f_i\},\ b\in\{t_{i+1},f_{i+1}\},\ i=1,\ldots,n-1 \Big\} \\
&\cup \left\{ (s,t_1), (s,f_1), (t_n,t), (f_n,t) \right\}.
\end{align*}
Note that $G'$ does not depend on the clauses of the given 3SAT instance, only on the number of variables. Figure~\ref{fig:metahard1} visualizes graph $G'$ for the case that $n=3$.

\begin{figure}[htbp]
\begin{center}
\includegraphics[width=0.6\textwidth]{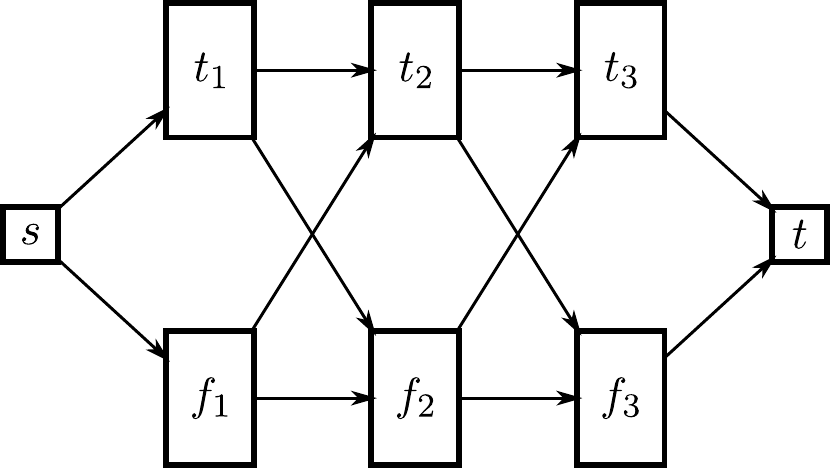}
\caption{Graph $G^\prime$ showing only representative nodes for each district.
}\label{fig:metahard1}
\end{center}
\end{figure}

Any meta-path thus consists of a sequence of nodes $t_i$ and $f_i$ to reach $t$ from $s$, which corresponds to a truth assignment of the Boolean variables of the 3SAT instance.

We now explain the construction of the (micro) graph $G$. Each node $t_i\in V'$ is replaced by two nodes $t^1_i$ and $t^2_i$, and each node $f_i\in V'$ is replaced by two nodes $f^1_i$ and $f^2_i$. Each pair $(f^1_i,f^2_i)$ and $(t^1_i,t^2_i)$ forms a districts. The set of edges is
\begin{align*}
E = & E_1 \cup E_2 \\
&\cup \{ (t^1_i,t^2_i) : i\in[n]\} \cup \{ (f^1_i,f^2_i) : i\in[n] \} \\
&\cup \{ (s,t^1_1), (s,f^1_1), (t^2_n,t), f^2_n,t) \} \\
\text{with } E_1 = &\Big\{ (a,b) : a\in\{t^1_i,f^1_i\},\ b\in\{t^1_{i+1},f^1_{i+1}\},\ i=1,\ldots,n-1 \Big\} \\
E_2 = & \Big\{ (a,b) : a\in\{t^2_i,f^2_i\},\ b\in\{t^2_{i+1},f^2_{i+1}\},\ i=1,\ldots,n-1 \Big\}.
\end{align*}
Finally, we introduce $m$ cost vectors $\bm{c}^j\in\{0,1\}^E$, each of them corresponding to a clause $C_j$. The costs for all edges in $E_1$ and $E_2$ are always zero. The cost of an edge $(t^1_i,t^2_i)$ is zero if $x_i$ is contained in $C_j$, otherwise, the costs are one. Analogously, the cost of an edge $(f^1_i,f^2_i)$ is zero if $\bar{x}_i$ is contained in $C_k$, otherwise, the costs are one.
Finally, the costs of edges $(s,t^1_1)$, $(s,f^1_1)$, $(t^2_n,t)$, and $(f^2_n,t)$ are always zero.

Figure~\ref{fig:metahard2} illustrates the construction of $G$ with the cost scenario that corresponds to the clause $(x_1\vee \bar{x}_2 \vee x_3)$. There are only three edges with cost one, indicated by a dotted arrow. These are the edges $(f^1_1,f^2_1)$, $(t^1_2,t^2_2)$, and $(f^1_3,f^2_3)$.

\begin{figure}[htbp]
\begin{center}
\includegraphics[width=0.7\textwidth]{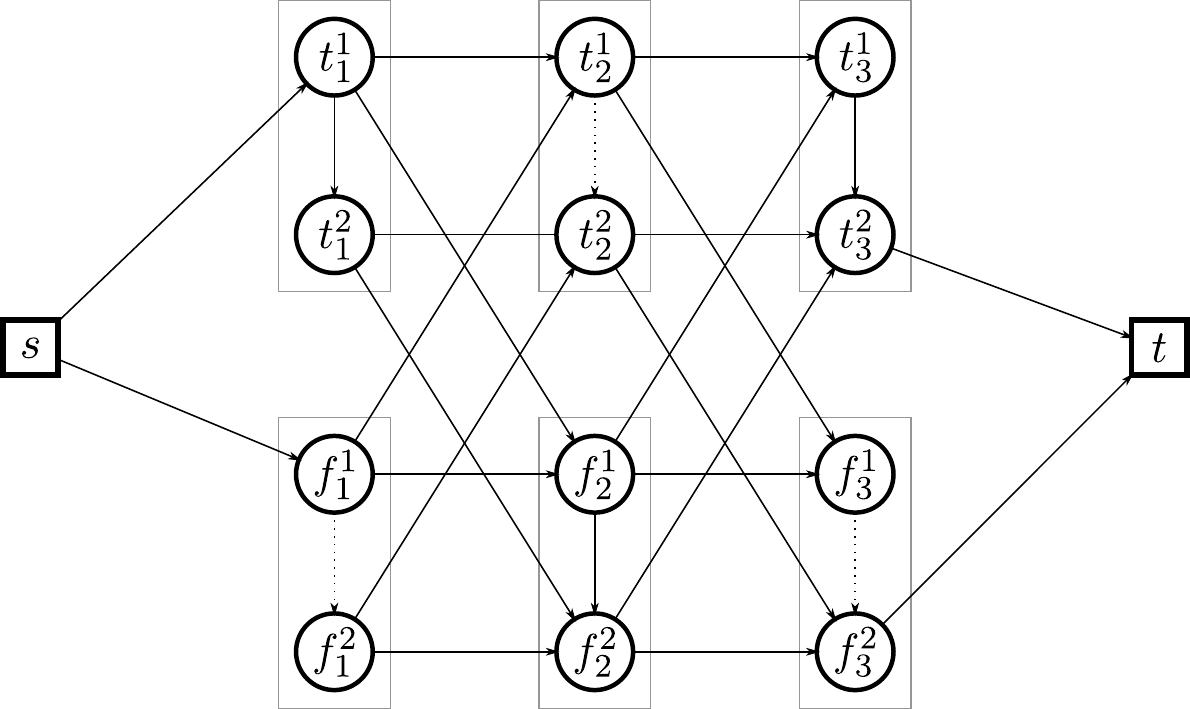}
\caption{Graph in proof of Theorem~\ref{th:metahard2}. Edges represented by solid lines have cost of zero, these represented by dotted lines have cost one.}\label{fig:metahard2}
\end{center}
\end{figure}

Note that any path begins with a node $t^1_1$ or $f^1_1$, but ends with a node $t^2_n$ or $f^2_n$. This means that at some point, exactly one edge $(t^1_i,t^2_i)$ or $(f^1_i,f^2_i)$ must be used in any path. This is possible with costs zero in scenario $\bm{c}^j$ if and only if the meta-path corresponds to a truth assignment that fulfills clause $C_j$. Hence, the 3SAT instance is a yes-instance if and only if there is a meta-path that costs zero in each scenario, which completes the proof.
\end{proof}

\newpage
\setcounter{page}{1}
\section{Content to Be Included in the Electronic Companion}
In this document, the results of extensive experiments on instances of knapsack problems are presented. All experiments were conducted with respect to the setting described in Sections~\ref{sub::ex::setting} and~\ref{subs::ex::knap}.

\subsubsection{Varying Number of Items}\label{subs::ex::knap::items}
First, we present experiments in which the number of items was varied. For these, the number of training scenarios $N$ was set to ten, and the number of solution features $F_S$ was set to four. Figures \ref{fig::ks::items-all-tr} and \ref{fig::ks::items-heu-tr} illustrate the scaled objective value on the training data versus the number of items provided per instance. In Figure~\ref{fig::ks::items-all-tr}, the solid green and red lines are partially obscured by the blue line. It can be observed that there is a negative correlation between the number of items and the objective for all trees with two and four leaves. This trend is more distinct for trees with two leaves. Moreover, the methods that generate trees with four leaves demonstrate considerably better performance. However, it is remarkable that those from \texttt{LH4} perform significantly worse than the other methods, including the benchmark. This can be explained by the fact that the trees are constructed with regard to the accuracy and not the objective function value of the knapsack problems. As illustrated in Figure \ref{fig::ks::items-heu-tr}, the performance of \texttt{LH2} and \texttt{LH4} as well as the performance of \texttt{M2M2} increases again as the number of items rises. Even with the performance of \texttt{M2M2} increasing, for 140 or more items, this method is not able to outperform \texttt{META1}. This means that for these instances, the tree structure cannot be used beneficially. Moreover, the quality of the solutions from \texttt{MICRO2} and \texttt{MICRO4} continues to decline. Notably, the performance differences between \texttt{MICRO2} and \texttt{M2M2} are considerably larger than those between \texttt{MICRO4} and \texttt{M2M4}.

\begin{figure}[htbp]
    \centering
    \begin{subfigure}[t]{0.49\textwidth}
         \centering
         \includegraphics[width=1\linewidth]{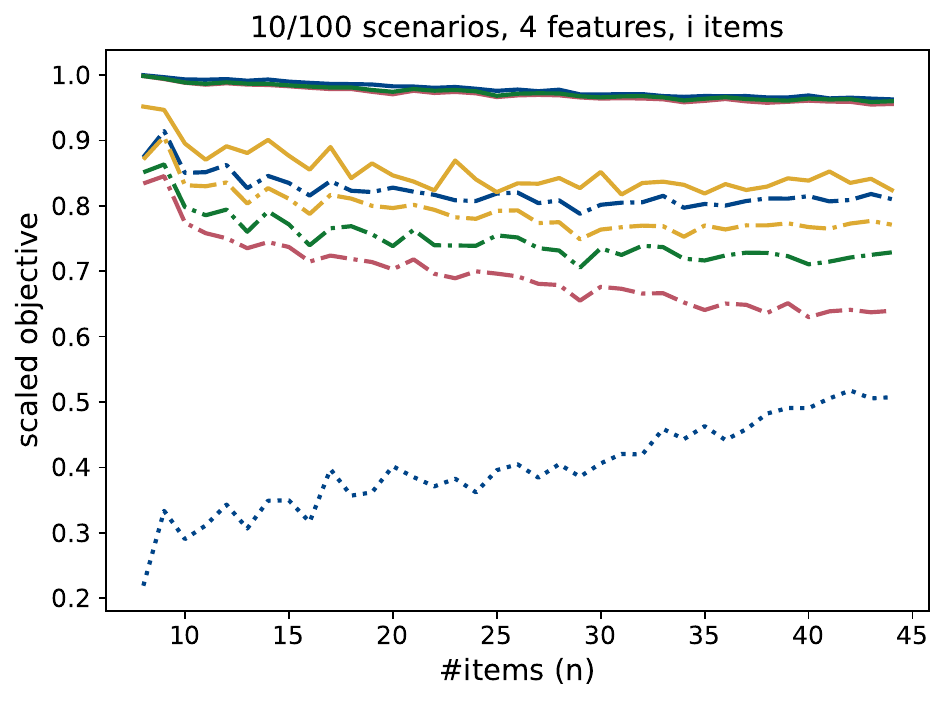}
         \caption{objective training/ \#items.}
         \label{fig::ks::items-all-tr}
    \end{subfigure}
    \begin{subfigure}[t]{0.49\textwidth}
         \centering
         \includegraphics[width=1\linewidth]{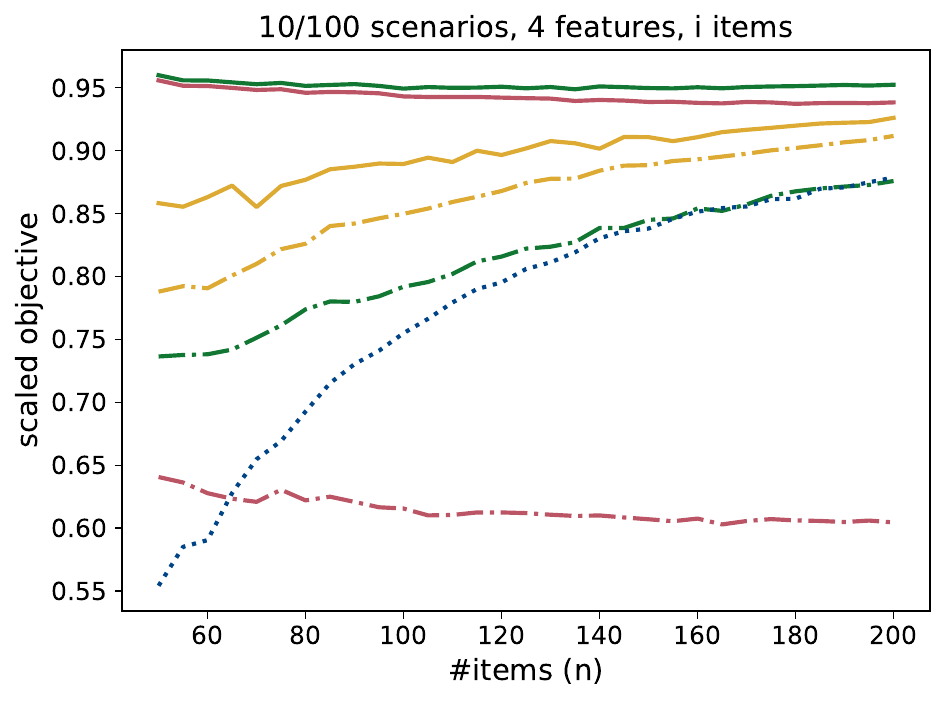}
         \caption{objective training/ \#items.}
         \label{fig::ks::items-heu-tr}
    \end{subfigure}
    \begin{subfigure}[t]{0.49\textwidth}
         \centering
         \includegraphics[width=1\linewidth]{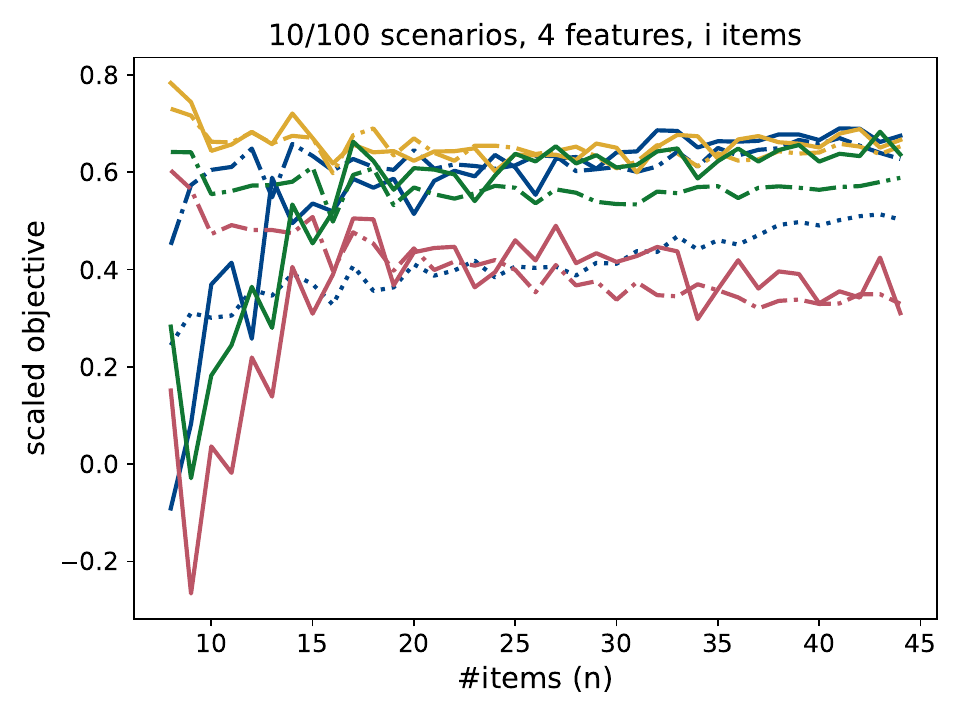}
         \caption{objective test/ \#items.}
         \label{fig::ks::items-all-te}
    \end{subfigure}
    \begin{subfigure}[t]{0.49\textwidth}
         \centering
         \includegraphics[width=1\linewidth]{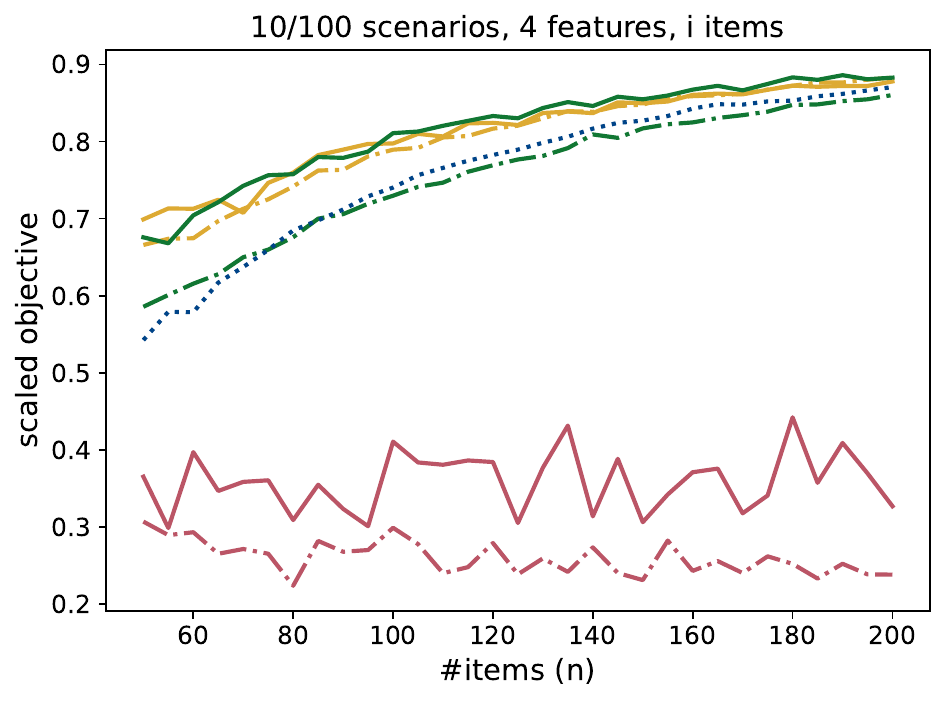}
         \caption{objective test/ \#items.}
         \label{fig::ks::items-heu-te}
    \end{subfigure}
    \begin{subfigure}[t]{0.49\textwidth}
         \centering
         \includegraphics[width=1\linewidth]{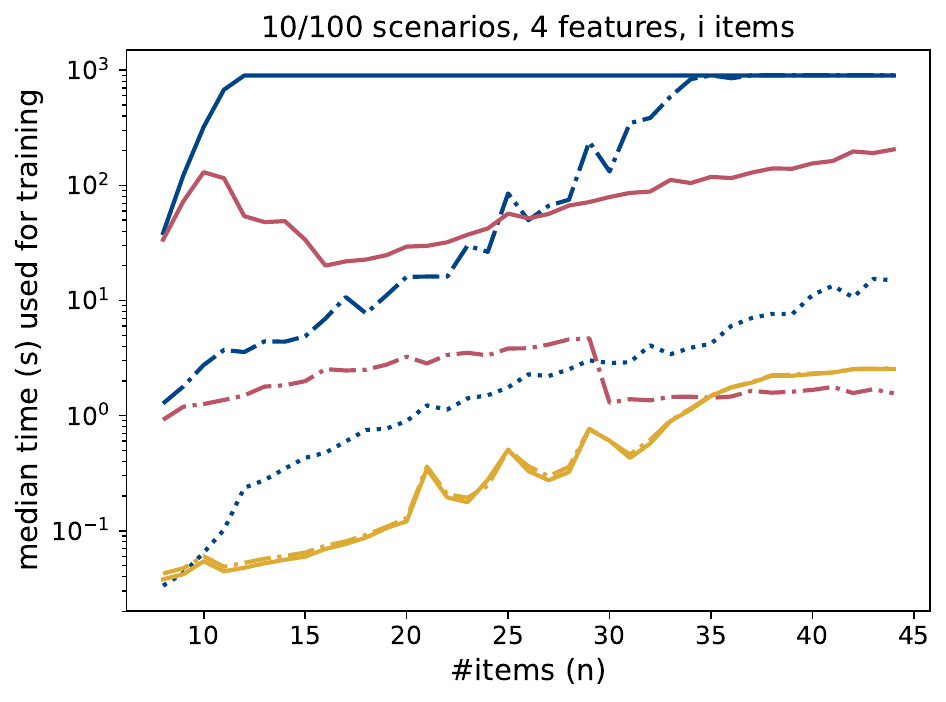}
         \caption{runtime/ \#items.}
         \label{fig::ks::items-all-ti}
    \end{subfigure}
    \begin{subfigure}[t]{0.49\textwidth}
         \centering
         \includegraphics[width=1\linewidth]{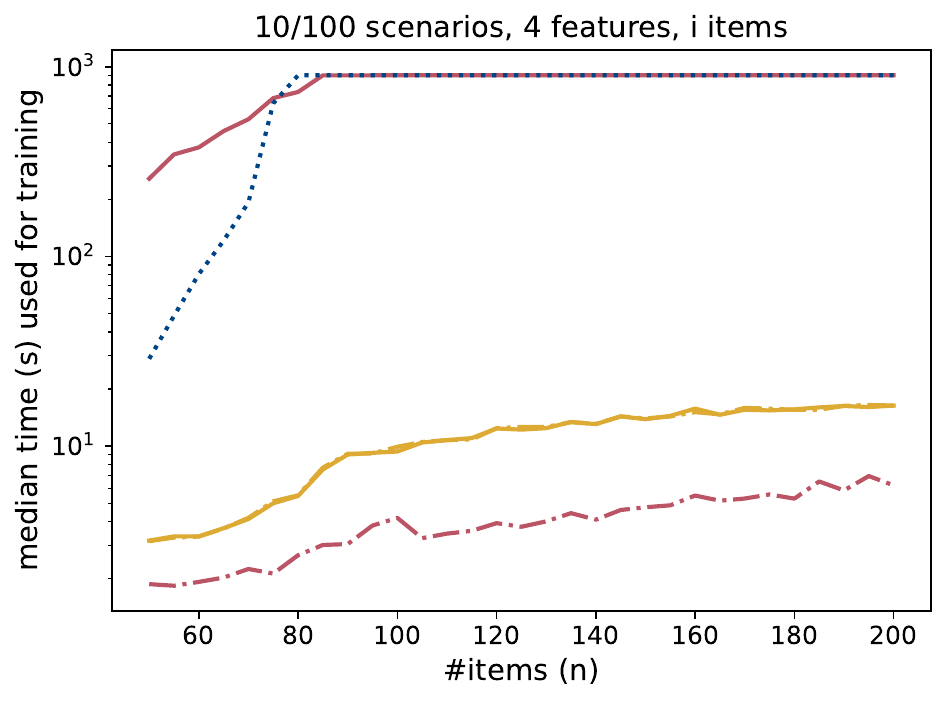}
         \caption{runtime/ \#items.}
         \label{fig::ks::items-heu-ti}
    \end{subfigure}
    \begin{subfigure}[t]{1\textwidth}
         \centering
         \vspace{.4cm}
         \small
\begin{tabular}{clclclcl}
    \raisebox{.3ex}{\includegraphics[scale=.05]{legend_blue_full.pdf}} & \texttt{MIP4} & \raisebox{.3ex}{\includegraphics[scale=.05]{legend_red_full.pdf}} & \texttt{MICRO4} & \raisebox{.3ex}{\includegraphics[scale=.05]{legend_yellow_full.pdf}} & \texttt{LH4} & \raisebox{.3ex}{\includegraphics[scale=.05]{legend_green_full.pdf}} & \texttt{M2M4}\\
    \raisebox{.3ex}{\includegraphics[scale=.05]{legend_blue_dasheddotted.pdf}} & \texttt{MIP2} & \raisebox{.3ex}{\includegraphics[scale=.05]{legend_red_dasheddotted.pdf}} & \texttt{MICRO2} & \raisebox{.3ex}{\includegraphics[scale=.05]{legend_yellow_dasheddotted.pdf}} & \texttt{LH2} & \raisebox{.3ex}{\includegraphics[scale=.05]{legend_green_dasheddotted.pdf}} & \texttt{M2M2}\\
    \raisebox{.3ex}{\includegraphics[scale=.05]{legend_blue_dotted.pdf}} & \texttt{META1} &
\end{tabular}

    \end{subfigure}
    \caption{Plots for knapsack and synthetic data, \#items on x-axis. A higher scaled objective indicates a better result. The solution-based approach, which serves as a benchmark, is highlighted in red.}
    \label{fig::ks::items}
\end{figure}

Figures \ref{fig::ks::items-all-te} and \ref{fig::ks::items-heu-te} depict the relationship between the scaled test objective and the number of items per instance. For all methods except \texttt{MICRO2} and \texttt{MICRO4}, a positive correlation can be identified. For instances with a small number of items, the methods \texttt{MIP4}, \texttt{MICRO4}, \texttt{M2M4} are overfitting. The gap between \texttt{MICROx} and \texttt{M2Mx} derived from it increases significantly with an increasing number of items. In contrast to its performance on the training data, the performance on the test data of \texttt{LH2} and \texttt{LH4} is highly competitive. For instances with many items, significantly better performance is achieved when using \texttt{META1} than \texttt{MICROx}. In general, it can be concluded that the performance of our methods appears to be clearly advantageous when evaluated on test data.

The median computing time used is shown in Figures \ref{fig::ks::items-all-ti} and \ref{fig::ks::items-heu-ti}. The time required for solving via \texttt{M2Mx} results almost exclusively from the time used to calculate \texttt{MICROx}. The data points of \texttt{MICROx}, therefore, also represent the computing time of \texttt{M2Mx}. Even small instances for \texttt{MIP4} could not be solved to optimality without hitting the time limit of 900 seconds. It can even be seen that all MIP-based methods except \texttt{MICRO2} cannot be solved to optimality within the time limit. In every run, feasible solutions were found within the time limit. \texttt{LHx} can be solved for all configurations within seconds. For both \texttt{MICRO2} and \texttt{MICRO4}, dips in the progressions can be observed.

\subsubsection{Varying Number of Scenarios}\label{subs::ex::knap::scens}

Figures \ref{fig::ks::scens-all-ti} and \ref{fig::ks::scens-heu-ti} show the computation time for a varying number of scenarios with $n=16$ and $F_S=4$. For \texttt{MIP4}, even instances with ten scenarios couldn't be solved to optimality. Also, when using \texttt{MIP2} and \texttt{MICRO4}, the time limit was hit at 16 and 18 training scenarios respectively. The computation times for \texttt{MICRO2} show a notable dip from 14 to 23 scenarios, which we were not able to explain. In contrast to the MIP-based approaches, when using \texttt{LH2} and \texttt{LH4}, we can generate decision trees in less than two seconds even for instances with 200 scenarios.

\begin{figure}[htbp]
    \centering
    \begin{subfigure}[t]{0.49\textwidth}
         \centering
         \includegraphics[width=1\linewidth]{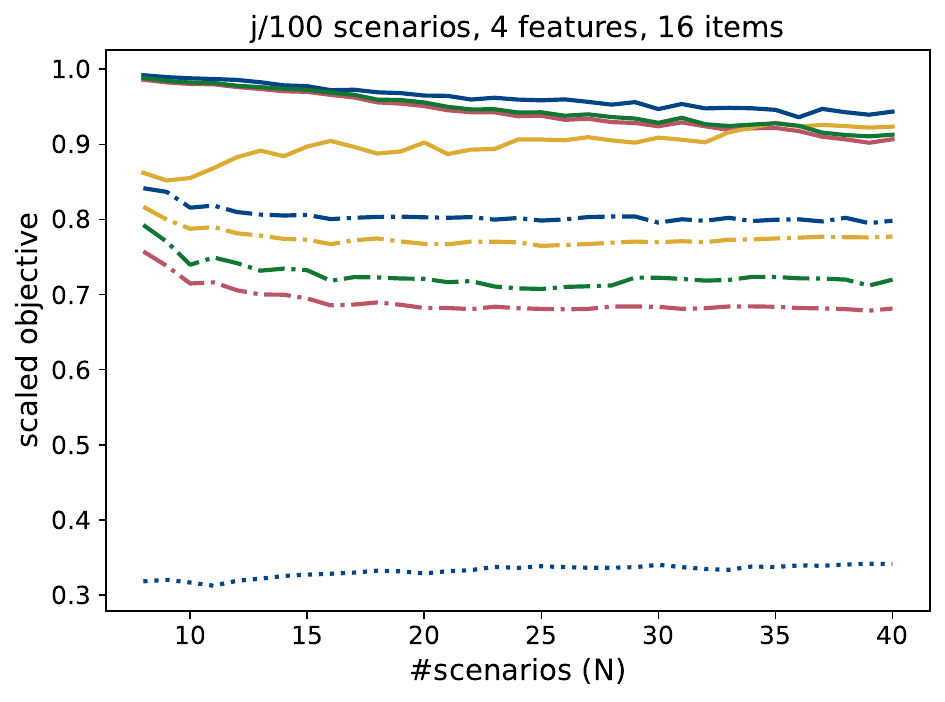}
         \caption{objective training/ \#scenarios.}
         \label{fig::ks::scens-all-tr}
    \end{subfigure}
    \begin{subfigure}[t]{0.49\textwidth}
         \centering
         \includegraphics[width=1\linewidth]{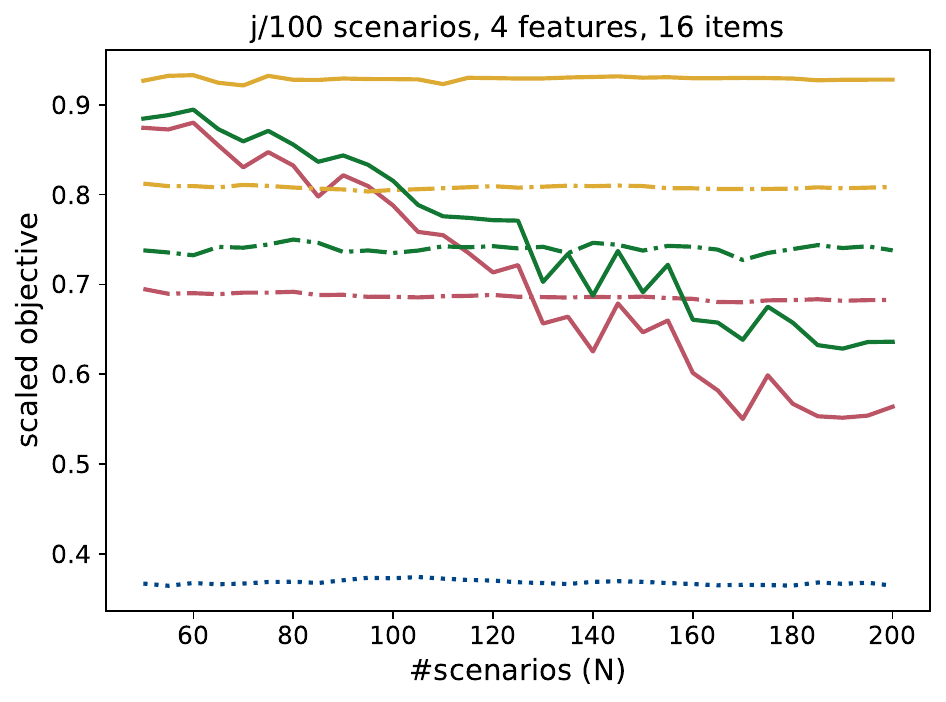}
         \caption{objective training/ \#scenarios.}
         \label{fig::ks::scens-heu-tr}
    \end{subfigure}
    \begin{subfigure}[t]{0.49\textwidth}
         \centering
         \includegraphics[width=1\linewidth]{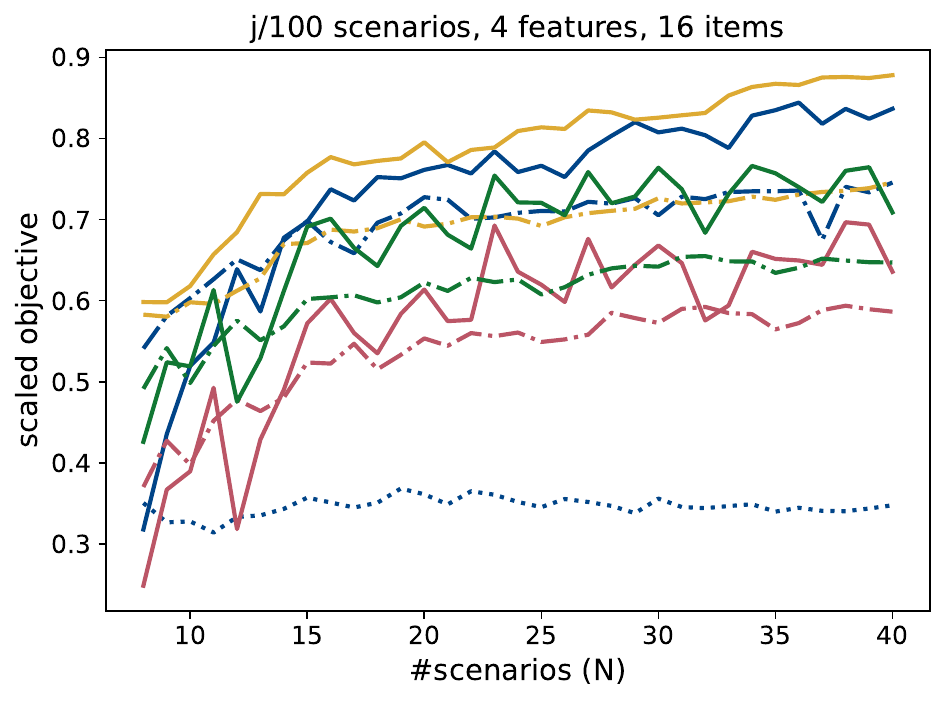}
         \caption{objective test/ \#scenarios.}
         \label{fig::ks::scens-all-te}
    \end{subfigure}
    \begin{subfigure}[t]{0.49\textwidth}
         \centering
         \includegraphics[width=1\linewidth]{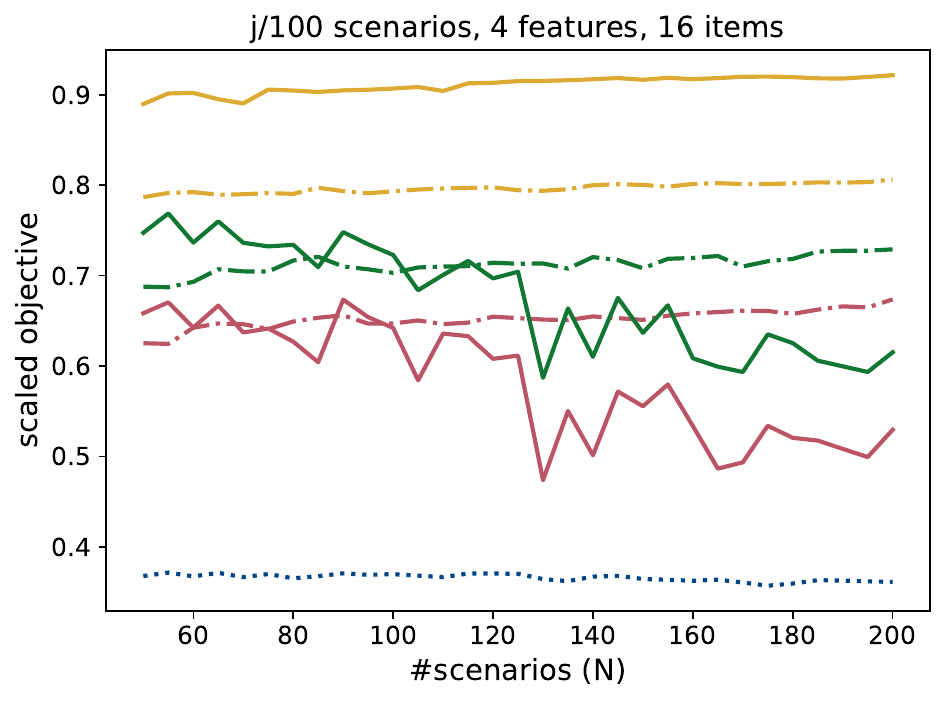}
         \caption{objective test/ \#scenarios.}
         \label{fig::ks::scens-heu-te}
    \end{subfigure}
    \begin{subfigure}[t]{0.49\textwidth}
         \centering
         \includegraphics[width=1\linewidth]{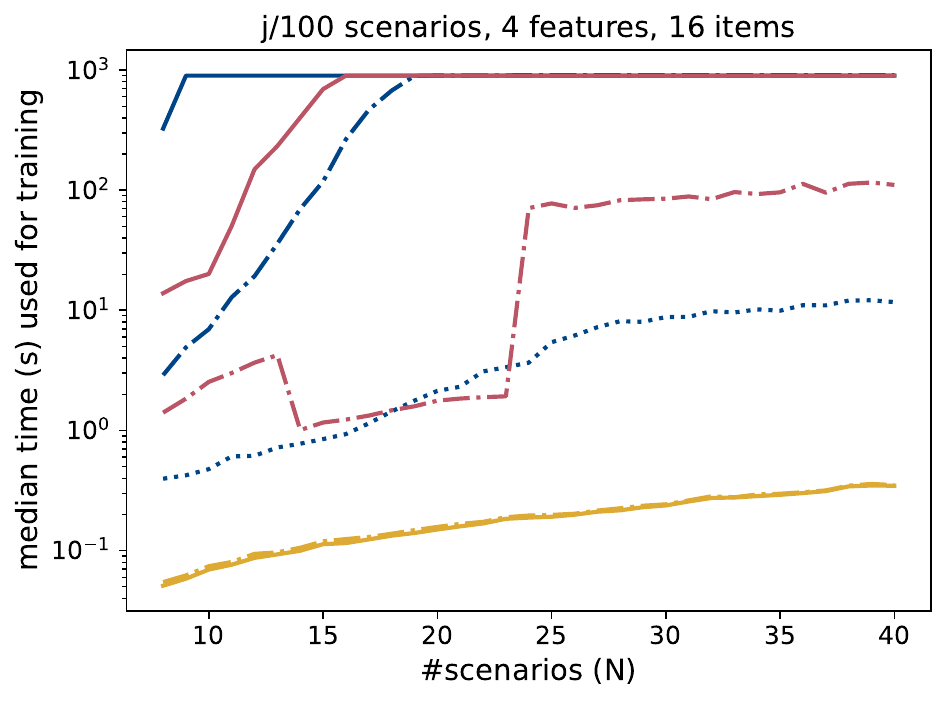}
         \caption{runtime/ \#scenarios.}
         \label{fig::ks::scens-all-ti}
    \end{subfigure}
    \begin{subfigure}[t]{0.49\textwidth}
         \centering
         \includegraphics[width=1\linewidth]{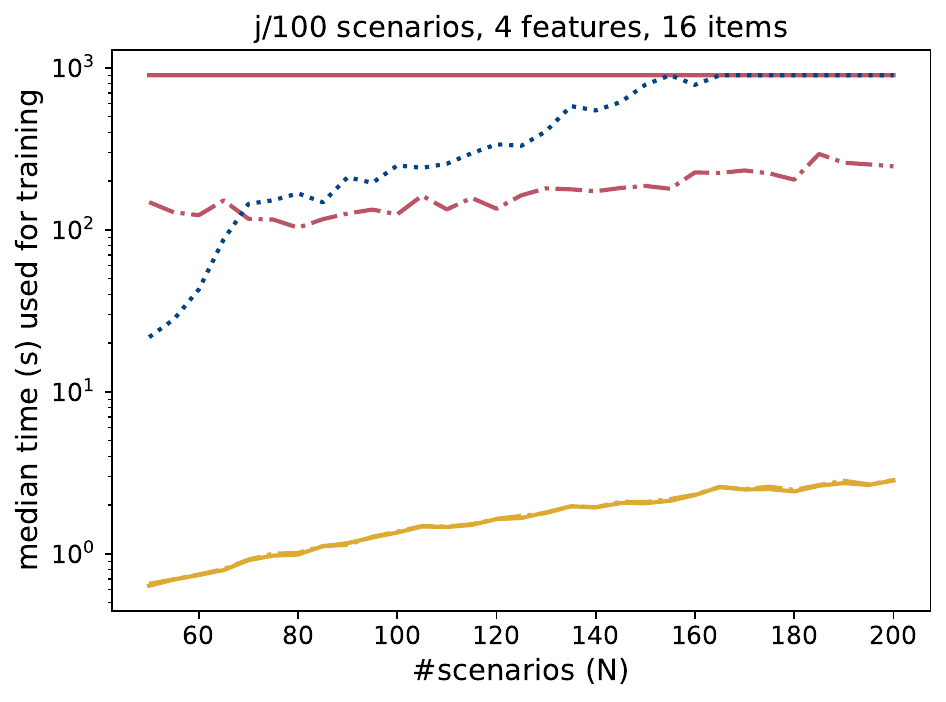}
         \caption{runtime/ \#scenarios.}
         \label{fig::ks::scens-heu-ti}
    \end{subfigure}
    \begin{subfigure}[t]{1\textwidth}
         \centering
         \vspace{.4cm}
         \small
\begin{tabular}{clclclcl}
    \raisebox{.3ex}{\includegraphics[scale=.05]{legend_blue_full.pdf}} & \texttt{MIP4} & \raisebox{.3ex}{\includegraphics[scale=.05]{legend_red_full.pdf}} & \texttt{MICRO4} & \raisebox{.3ex}{\includegraphics[scale=.05]{legend_yellow_full.pdf}} & \texttt{LH4} & \raisebox{.3ex}{\includegraphics[scale=.05]{legend_green_full.pdf}} & \texttt{M2M4}\\
    \raisebox{.3ex}{\includegraphics[scale=.05]{legend_blue_dasheddotted.pdf}} & \texttt{MIP2} & \raisebox{.3ex}{\includegraphics[scale=.05]{legend_red_dasheddotted.pdf}} & \texttt{MICRO2} & \raisebox{.3ex}{\includegraphics[scale=.05]{legend_yellow_dasheddotted.pdf}} & \texttt{LH2} & \raisebox{.3ex}{\includegraphics[scale=.05]{legend_green_dasheddotted.pdf}} & \texttt{M2M2}\\
    \raisebox{.3ex}{\includegraphics[scale=.05]{legend_blue_dotted.pdf}} & \texttt{META1} &
\end{tabular}

    \end{subfigure}
    \caption{Plots for knapsack and synthetic data, \#scenarios on x-axis. A higher scaled objective indicates a better result. The solution-based approach, which serves as a benchmark, is highlighted in red.}
    \label{fig::ks::scens}
\end{figure}

The relationship between the scaled objective on the training set and the number of scenarios used for training is shown in Figures \ref{fig::ks::scens-all-tr} and \ref{fig::ks::scens-heu-tr}. For all methods generating trees with two leaves, the influence on the training performance is marginal. The performance of \texttt{MICRO4} and \texttt{M2M4} decreases significantly when more than 60 scenarios are used, and even falls below that of \texttt{MICRO2} and \texttt{M2M2}. This can be explained because no (proven) optimal solutions could be determined for \texttt{MICRO4}. It can be seen that the quality of the solutions of \texttt{MICRO4} and \texttt{M2M4} is strongly correlated and that of \texttt{M2M4} presumably depends on \texttt{MICRO4}.

In Figure \ref{fig::ks::scens-all-te}, it can be seen that for all methods utilized, the performance increases with the number of scenarios used for training. In contrast, Figure \ref{fig::ks::scens-heu-te} shows that the performance of \texttt{MICRO4} and \texttt{M2M4} begins to decline at a certain point as the number of scenarios increases. As with the data in Fig. \ref{fig::ks::scens-heu-tr}, this is due to the intractability of the models. It is noteworthy that \texttt{LH4} clearly outperforms \texttt{MIP4}. In general, all developed methods outperform the benchmark model of the same size. Furthermore, it can be seen that the marginal benefit of further training scenarios decreases.

\end{document}